\documentclass[12pt]{amsart}
\usepackage{amsmath,amsfonts,amsthm,amssymb,stmaryrd,comment}
\usepackage{xcolor}

\usepackage{array} 
\usepackage{url}
\newcolumntype{C}[1]{>{\centering\arraybackslash}m{#1}}

\definecolor{oblue}{RGB}{0,80,160}

\usepackage[
colorlinks=true,
linkcolor=oblue,
citecolor=oblue,
urlcolor=oblue
]{hyperref}
\usepackage{geometry}
\usepackage{etoolbox}
\patchcmd{\section}{\scshape}{\bfseries}{}{}
\makeatletter
\renewcommand{\@secnumfont}{\bfseries}
\makeatother

\renewcommand{\leq}{\leqslant}
\renewcommand{\geq}{\geqslant}

\newcommand*{\bbe}{
	\mathop{
		\mathchoice{\vcenter{\hbox{\larger[4]$\mathbb{E}$}}}
		{\kern0pt\mathbb{E}}
		{\kern0pt\mathbb{E}}
		{\kern0pt\mathbb{E}}
	}\displaylimits
}

\newcommand{\Z}{\mathbb{Z}}

\newcommand{\codim}{\operatorname{codim}}

\DeclareMathOperator{\ev}{ev}
\DeclareMathOperator{\Ker}{Ker}
\DeclareMathOperator{\Supp}{Supp}
\DeclareMathOperator{\Gen}{Gen}

\newtheorem{theorem}{Theorem}[section]
\newtheorem{lemma}[theorem]{Lemma}

\newtheorem{corollary}[theorem]{Corollary}

\newtheorem{proposition}[theorem]{Proposition}

\newtheorem{theoremx}{Theorem}
 
\newtheorem{conjecturex}[theoremx]{Conjecture}

\theoremstyle{definition}
\newtheorem{definition}[theorem]{Definition}
\newtheorem{example}[theorem]{Example}
\newtheorem{remark}[theorem]{Remark}
\newtheorem{remarks}[theorem]{Remarks}

\newcommand{\divi}[1]{\langle #1 \rangle}

\usepackage{caption}
\usepackage{subcaption}
\usepackage{tikz}
\usepackage[shortlabels]{enumitem}

\usepackage[backend=bibtex,style=numeric, maxbibnames=99]{biblatex}
\title
{High dimensional Riemann--Roch spaces in linear spaces with small squares}

\author{Nihan Tan\i sal\i~} 
\author{Mieke Wessel}
\address{
	Centre de recherche INRIA Saclay,
	\'Ecole Polytechnique, Institut Polytechnique de 
	Paris, 1 rue Honor\'e d’Estienne d’Orves, 91120 Palaiseau Cedex
	France}
\email{nihan.tanisali@inria.fr}
\email{mieke.wessel@inria.fr}
\date{}
\begin{document}
	
	\begin{abstract}
		Let $F$ be a function field over an algebraically closed field $K$ and $S$ a finite dimensional $K$-subspace of $F$. The square of $S$ is spanned by all products of pairs of elements in $S$. We conjecture that if $\dim S^2 \leq 3 \dim S - 4$, then $S^2$ must contain a Riemann--Roch space of dimension at least $2 \dim S - 1 + g$, where $g$ is the genus of $F$. This generalizes a theorem of Freiman from additive combinatorics, stating that small sumsets must contain long arithmetic progressions. We prove our conjecture in the case that $S$ is contained in a Riemann--Roch space of dimension at most $3/2\dim S + 1$. For the proof we study the annihilator of $S$ and introduce the notion of weight for linear forms. 
	\end{abstract}
	
	\maketitle
	\section{Introduction}
	Let $K$ be a field and $F/K$  be a field extension. For finite dimensional  $K$-subspaces $U,V \subseteq F,$ we define their product by 
	$$U V := \langle uv  \mid u \in U ,v \in V  \rangle. $$
	In particular when $U=V=S$, we write $S^2$ for the square of $S$, defined by  
	$ S^2 := \langle s_1s_2 \mid s_1, s_2 \in S \rangle.$
	We are interested in structural properties of $S$ and $S^2$ when $\dim S^2$ is small. These types of questions are generalizations of inverse theorems from additive combinatorics, where given finite sets $A$ and $B$ in an abelian group $G$ and their sumset $$A + B = \{ a + b \mid a \in A, b\in B \},$$ they prove structural properties on $A$ and $B$ whenever $|A+B|$ is small. More precisely, motivated by an additive combinatorial result of Freiman \cite{Freiman2009intcase}, we ask whether $S^2$ must contain a large Riemann--Roch space when $F$ is a function field and $K$ is algebraically closed.
	
	To see that the extension field set-up is really a generalization of the additive one, note that given $A \subset G$ and $x \in F \backslash K$ we can define $$S_A := \langle x^a \mid a \in A \rangle,$$ 
	in which case $S_AS_B = S_{A+B}$ and $\dim_K S_A = |A|$. 
	
	The first significant result in the study of multiplicative analogues of addition theorems is due to Hou, Leung and Xiang, who proved an extension field version of Kneser's classical addition theorem under a separability condition on $F$ \cite{HLX02,Kne56}. Bachoc, Serra and Z\'emor later removed this separability assumption, extending the result to arbitrary field extensions \cite{BSZ18}. In the slightly more general setting of skew field extensions Eliahou and Lecouvey generalized several other results from additive combinatorics, including theorems of Kemperman and Olson \cite{eliahou09someadd} and Pl\"unnecke's inequalities \cite{lecouveyplunnecke}. Beck and Lecouvey further developed these methods in the even more general setting of associative unital algebras, obtaining generalizations of sumset results for non-abelian groups \cite{BL17}. 
	
	The study of product spaces is also motivated by coding theory, for if $K$ is a finite field and $F \cong K^n$, the $K$-subspaces $U$ and $V$ are isomorphic to linear codes and their product can be interpreted as a component wise product. Products of codes are used for several applications in cryptography, such as linear secret sharing schemes and encryption schemes. Randriambololona studied products and higher powers of linear codes with a focus on their dimensions and minimum distances \cite{Ran15}. Mirandola and Z\'emor specifically analyzed the structural properties of pairs of codes with product of small dimension \cite{MZ15}.
	
	Assuming that $K \subseteq U, V$ and that $U$ and $V$ are not both contained in some intermediate field extension $L \subsetneq F$, it holds for general $F$ that $$\dim_K UV\geq \min([F:K], \dim_K U + \dim_K V - 1).$$ 
	The case where $\dim_K UV = \dim_K U + \dim_K V - 1$ and $[F:K] = p$, for some prime $p$ was studied by Bachoc, Serra and Z\'emor leading to a linear analogue of Vosper's Theorem \cite{BSZ17}. In the symmetric case, where $U = V = S$, we can ask how the structure of $S$ changes when 
	$$\dim_K S^2=2\dim_K S-1+\gamma \qquad \text{with} \qquad \gamma \geq 0.$$
	This question was initiated by Bachoc, Couvreur and Z\'emor in \cite{bachoc2016freiman} taking $F$ to be a function field and it is motivated by Freiman's $3k-4$ Theorem \cite{freiman1973foundations}. The cases $\gamma=0, 1$ were completely characterized in \cite{bachoc2016freiman}. For rational function fields $F = K(x)$, the next case, $\gamma=2$, was classified in \cite{Wes24} by the second author. 
	
	The classical $3k-4$ Theorem states that for a finite set $B \subset \Z$ such that
	$$|B+B|=2|B|-1+\gamma_B
	\qquad\text{with}\qquad
	0\leq\gamma_B \leq |B|-3,$$
	it must hold that $B$ is contained in an arithmetic progression of length at most $|B|+\gamma$. Note that the assumption is true for some $\gamma_B$ if and only if $|B+B| \leq 3|B|-4$ (hence $3k-4$). The complete function field analogue of this was proved by Couvreur and Z\'emor \cite{couvreur2024freimanproof}. It implies in particular that the natural generalization of arithmetic progressions are Riemann--Roch spaces. 
	\begin{theorem}\label{thm:3k-4}
		Let $K$ be a perfect field and let $F$ be an extension field of $K$ in which
		$K$ is algebraically closed; that is, any element of $F$ that is algebraic over
		$K$ belongs to $K$. Let $S$ be a finite-dimensional $K$-subspace of $F$ such
		that $K \subseteq S$ and $F = K(S)$. If
		$$ \dim_K S^2 = 2\dim_K S - 1 + \gamma \qquad\text{
			with}\qquad 0 \leq \gamma \leq \dim_K S - 3,$$
		then $F$ has transcendence degree $1$ over $K$, and its genus $g$ satisfies
		$$g \leq \gamma.$$
		Furthermore, there exists a Riemann--Roch space $L(D)$ of $F$ such that
		$S \subseteq L(D)$ and
		$$\dim_K L(D) \leq \dim_K S + \gamma - g. $$
	\end{theorem}

	More than thirty years after proving his $3k-4$ Theorem, Freiman also proved another result about arithmetic progressions and sumsets under the same assumptions on $B$ \cite{Freiman2009intcase}. 
	
	\begin{theorem}\label{thm:frei}
		Let $B \subseteq \Z_{\geq 0}$ be a finite set such that $0 \in B$ and $\gcd(B) = 1$. Suppose that 
		$$|B + B| \leq 3|B| - 4.$$ 
		Then there exists an interval $J \subseteq B+B$ such that $|J| \geq 2k-1$. 
	\end{theorem}
	\begin{remark}
		The actual theorem Freiman proved in \cite{Freiman2009intcase} is slightly stronger and more technical. It states that if $J = (e, c+\max(B))$ and $B \cap (e, c)$ has $d$ holes, then $|J| \geq 2k - 1 + 2d$. Here, holes are elements that are in $(e, c)$ but not in $B \cap (e, c)$. 
	\end{remark}
	
	In the analogy described above, the sumset $B+B$ is replaced by the product space $S^2$, while intervals or arithmetic progressions are replaced by Riemann--Roch spaces. Thus, the existence of a long interval contained in $B+B$ suggests asking whether $S^2$ contains a Riemann--Roch space of large dimension. This leads to the following conjecture.
	
	\begin{conjecturex}\label{Conj:main}
		Let $K$ be an algebraically closed field and let $F$ be a non-trivial extension field of $K$. Let $S \ni s$ be a finite-dimensional $K$-subspace such that $F = K(s^{-1}S)$ and $\dim_K S^2 \leq 3\dim_K S - 4$. Then there exists a Riemann--Roch space $L(G)$ such that 
		\begin{align*}
			L(G) \subseteq S^2 \quad \text{ and } \quad \dim_K L(G) \geq 2\dim_K S + g - 1. 
		\end{align*}
	\end{conjecturex}
	
	The main result of this paper is that the conjecture holds when $S$ has sufficiently small co-dimension in $L(D)$, the minimal Riemann--Roch space containing it. 
	
	\begin{theoremx}\label{Thm:main}
		Let $K, F$ and $S$ be as in Conjecture \ref{Conj:main} and $D$ a divisor of $F$ such that $S \subseteq L(D)$. Define $b$ as the co-dimension of $S$ in $L(D)$ and assume $2b \leq \dim_K S + 2$. Then Conjecture \ref{Conj:main} holds. 
	\end{theoremx}
	\begin{remark}
		As a consequence of Theorem \ref{thm:3k-4} the general range for $b$ is $[0, \dim_KS - 3 - g]$. 
	\end{remark}
	
	To prove Theorem \ref{Thm:main} we study the annihilators of $S$ and $S^2$ in the dual spaces $L(D)^\ast$ and $L(2D)^\ast$, or in other words we study the linear forms that disappear on $S$ and $S^2$. Recall that any linear form can be written as a linear combination of evaluation maps, that is maps that send $f$ to $f(P)$ for some fixed place $P$. We generalize this to so called \textit{evaluation forms}, by also considering maps that send $f$ to the evaluation at $P$ of a (higher order) Hasse-derivative of $f$. Then, we introduce the notion of \textit{weight}, which for each linear form roughly measures the minimal number of evaluation forms needed in a linear combination. This idea was inspired by the original proof of Freiman, where he studies the missing elements or \textit{holes} of $B$ and $B+B$ and by a conjecture of Randriambololona \cite[Subsection 8.3]{ran20quadhull}. The conjecture states that a hypersurface $H$ of a Riemann--Roch space $L(A)$ does not have full square $L(A)^2 = L(2A)$ if and only if $H$ is one of the following:
	\begin{enumerate}[\normalfont (i)]
		\item For some place $P$ of $F$ we have $H = L(A-P)$. 
		\item There are places $P$ and $Q$ of $F$ and $\alpha \in K^\times$ such that for all $f \in H$ we have $f(P) = \alpha f(Q)$. 
		\item There is a place $P$ of $F$ and $\alpha \in K$ such that for all $f \in H$ we have $D^{(1)}f(P) = \alpha f(P)$, where $D^{(1)}f$ denotes the first Hasse-derivative of $f$. 
	\end{enumerate}
	Randriambololona proves one direction, namely that any such hypersurface does indeed not have full square. Our result in Proposition \ref{prop:m_dimension} below implies the other direction and moreover, gives a necessary condition for a space $U \subseteq L(A)$ of any fixed co-dimension to not have full square. Besides proving Theorem \ref{Thm:main} we also further develop the underlying theory of evaluation forms and their weights, which we believe will be useful to prove other linear analogues of results from additive combinatorics in the future.
	
	The paper is structured as follows. In Section \ref{sec:pre} we recall some preliminaries about function fields and define for a Riemann--Roch space $L(A)$ the concepts of evaluation forms, formal expressions of evaluation forms and weight. In Section \ref{sec:lems} we build upon this and prove several results that hold for general $U \subset L(A)$. The first time we take into account that $\dim_K S^2 \leq 3\dim_K S - 4$ is in Section \ref{Sec:specific}. Here we give a short overview of Frieman's proof of Theorem \ref{thm:frei} and we prove that the conjecture must hold for $2b < \dim_K S$. Finally, in Section \ref{sec:main}, we prove Theorem \ref{Thm:main} for the full range $2b \leq \dim_K S + 2$.
	
	\subsection*{Notation}
	For the rest of the paper we will make the following assumptions. 
	
	Let $K, F$ and $S$ be as in Conjecture \ref{Conj:main}, that is $K$ is algebraically closed, $F$ is a function field over $K$ and $S\subset F$ such that for $s\in S$ we have $K(s^{-1}S) = F$. Define $D$ as the divisor of minimal degree such that $S \subseteq L(D).$ 
	
	Note that $L(G) \subset S$ implies that $fL(G) = L(G - (f)) \subset fS$ for any $f \in F^\times$, so after possibly replacing $S$ by $s^{-1}S$ and $D$ by $D + (s)$ for some $s\in S$, we may assume that
	$$K\subseteq S\subseteq L(D)$$
	and in particular that $ D\ge 0$.
	
	Finally, we define the following constants: 
	\begin{align*}
		k&:= \dim_K S;\\
		N &:= \deg (D); \\
		\gamma &:= \dim_K S^2 - 2\dim_K S + 1= \dim_K S^2 - 2k + 1; \\
		b &:= \dim_K L(D)  - \dim_K S.
	\end{align*}
	The constant $\gamma$ is also called the \textit{combinatorial genus} of $S$, and $b$ can be recognized as the co-dimension of $S$ in $L(D)$. We will see that $\dim_K L(D) = N - g + 1$ and thus, $N = b + k + g - 1.$
	
	By assumption we have that $\gamma \leq k-3$, which by Theorem \ref{thm:3k-4}  implies that $\dim_K L(D) \leq k + \gamma$ and that $F$ has transcendence degree $1$. 
	
	\subsection*{Acknowledgements}
	We would like to thank Alain Couvreur and Gilles Z\'emor for helpful discussions and for bringing our attention to this problem. The first author was funded by the Horizon-Europe MSCA-DN project ENCODE and the second author was funded by the HYPERFORM consortium, funded by France through Bpifrance.
	
	\section{Preliminaries and evaluation forms}\label{sec:pre}
	\subsection{Preliminaries}
	We first recall a few standard notions and facts about function fields, Riemann–Roch spaces and linear algebra that will be used throughout the paper.
	
	Fix a field $K$. We take all vector-space dimensions over $K$ unless otherwise specified, and write $\dim$ instead of $\dim_K$.
	
	A \textit{function field} over $K$ of transcendence degree $1$, denoted by $F/K$, is a finite extension of the rational function field $K(x)$, where $x$ is transcendental over $K$. 
	
	A \textit{valuation ring} $\mathcal{O}$ of $F/K$ is a subring such that $K\subsetneq \mathcal O\subsetneq F$ and for every $z\in F^\times$, either $z\in\mathcal O$ or $z^{-1}\in\mathcal O$. Every valuation ring $\mathcal O$ has a unique maximal ideal $P$ which is principal. We call $P$ a \textit{place} of $F/K$. A generator $t_P$ of a place $P$ is called a \textit{local parameter}. The corresponding \textit{discrete valuation} is denoted by
	$v_P:F^\times\longrightarrow\mathbb Z$ and gives $v_P(f) = m$ when $f = at_P^{m} + O(t_P^{m+1})$ for non-zero $a$.    
	
	The \textit{divisor group} of $F$, denoted by $\operatorname{Div}(F)$, is the free abelian group generated by the places of $F$. Its elements are called divisors; thus, a \textit{divisor} is a finite formal sum of places.
	For a place $P$ of $F$, its \textit{degree} is defined by
	$ \deg P:=[\mathcal O_P/P:K]. $
	Accordingly, the degree of a divisor
	$$ A=\sum_P n_PP $$
	is defined by
	$$ \deg A:=\sum_P n_P\deg P. $$
	In the following,  we assume that $K$ is algebraically closed. Therefore, any place $P$ has degree 1 and any divisor $A$ has degree $\sum_{P} n_P$. The \textit{valuation of $A$ at a place $P$} is defined by $v_P(A):=n_P$, and the \textit{support} of $A$ is defined by $\Supp(A):=\{P\mid v_P(A)\neq 0\}.$
	
	For a function $f\in F^\times$, the \textit{principal divisor} associated with $f$ is defined by
	$$(f):=\sum_P v_P(f)P,$$
	where the sum runs over all places $P$ of $F$.
	This sum is finite, since $v_P(f)=0$ for all but finitely many places $P$. 
	We respectively write the \textit{zero divisor} and \textit{pole divisor} of $f$ as
	$$(f)_0 :=\sum_{ v_P(f)> 0} v_P(f) P, \qquad\text{and}\qquad (f)_\infty :=\sum_{ v_P(f)< 0} - v_P(f) P. $$
	We have $\deg (f)_\infty =\deg (f)_0$, and
	if $f\in F\setminus K$ then $ [F: K (f) ] =\deg (f)_\infty =\deg (f)_0$. 
	Consequently the degree of a principal divisor is $\deg((f)) =0 $. 
	
	The divisor group is partially ordered by coefficient-wise comparison. A divisor $ G=\sum_{i=1}^\sigma n_iP_i $
	is called \textit{effective} when $ n_i\geq 0$ for every $i=1,\ldots,\sigma$, this is denoted by $G \geq 0$. 
	Likewise, $G \ge H$ if $G - H \ge 0$.  
	Given a divisor $D$ of $F$, the \textit{Riemann--Roch space} $ L(D)$ is defined as
	$$ L(A) := \{ f \in F \mid (f) + A \ge 0 \} \cup \{0\}.$$
	In particular, we have $ L( 0)=K$.
	We denote the dimension of ${L} (A)$ by $\ell(A)$. 
	For every divisor $A$ and every function $f\in F^\times$, we have $L(A-(f) ) =fL(A) $. The dual space of $L(A)$ is denoted by $L(A)^\ast$.
	
	To any function field we can associate a smooth curve over $K$ that has $F$ as its function field. The \textit{genus} of $F$ is then defined as the genus of the curve and denoted by $g$.
	
	For a detailed introduction to function fields, places, and valuations, we refer the reader to \cite{stichtenoth2009algfuncfields}.
	
	Lastly, we recall the following three theorems on Riemann--Roch spaces.
	\begin{theorem}[{\cite[Thm.~1.5.17]{stichtenoth2009algfuncfields}}]\label{thm:RR}
		Let $F/K$ be a function field of genus $g$ and let $A$ be a divisor of $F/K$.
		If $\deg(A)\ge 2g-1$, then
		$$\ell(A)=\deg(A)+1-g.$$
	\end{theorem}
	
	We state Clifford's theorem in an extended form, tailored to the range needed in the paper.
	
	\begin{theorem}[Clifford's Theorem, {\cite[Thm.~1.6.13]{stichtenoth2009algfuncfields}} ]
		For all divisors $A$ with $-2\le \deg A \le 2g$ it holds that
		$$ \ell (A) \le 1+\frac{1}{2} \deg A . $$
	\end{theorem}
	
	\begin{theorem}[{\cite[Thm.~6]{Mumford2011}}]
		Let $A$ and $A'$ be two divisors of $F$ such that $\deg(A) \geq 2g$ and $\deg(A') \geq 2g+1$. Then,
		$$L(A)L(A') = L(A + A').$$ 
	\end{theorem}
	
	We will mostly want to apply these theorems for $A = D$ and $A= 2D$. The following direct consequence of Clifford's Theorem shows that this is permitted. 
	
	\begin{proposition}
		It holds that $\deg D \geq 2g+2$. 
	\end{proposition}
	\begin{proof}
		Suppose that $\deg D \leq 2g$, by Clifford's Theorem we have 
		$k \leq \ell(D) \leq 1 + g.$ Using Theorem \ref{thm:3k-4} we have $1 + g \leq 1 + \gamma \leq k - 2$, a contradiction. 
		
		Now suppose that $\deg D = 2g + 1$, then by Theorem \ref{thm:RR} we have that 
		$$ k \leq \ell(D) = g + 2 \leq \gamma + 2 \leq k - 1,$$ again a contradiction.
	\end{proof}
	
	\subsection{New definitions} \label{Sec:defs}
	Our main strategy is to study the subset $S\subseteq L(D)$ via the subspaces of linear forms in $L(D)^\ast$ and $L(2D) ^\ast$ that vanish on $S$ and $S^2$ respectively. To do this we set up a general framework in which we can express linear forms in so called evaluation forms. It will turn out that the more complex a linear form is, the more constraints there are for it to vanish on $S^2$. The complexity of a linear form can be derived from its weight, which we will also introduce below.
	
	In this entire subsection $A$ will be an arbitrary divisor of $F$. 
	
	\begin{definition}
		Let $U \subseteq L(A) $ be a linear subspace. The \textit{annihilator of $U$} is the following $K$-linear subspace
		$$\Phi_A(U) := \{ \varphi \in L(A)^\ast \mid \varphi(U) = 0 \}.$$ 
	\end{definition}
	
	\begin{remarks}~
		\begin{itemize}
			\item We can recover $U\subseteq L(A)$ from $\Phi_A(U)$ by 
			$$U = \{ f \in L(A) \mid \varphi(f) = 0 \textup{ for all } \varphi \in \Phi_A(U) \}.$$ 
			\item Let $U$ and $V$ be two subspaces of $L(A)$. Then, we have the inclusion $U \subseteq V$ if and only if  $\Phi_A(V) \subseteq \Phi_A(U)$. 
		\end{itemize}
	\end{remarks}
	
	\begin{definition}\label{def:functionals}
		Let $P$ be any place of $F$, and fix a local
		parameter $t$ at $P$. 
		For $f \in L(A)$ we have the local expansion
		$$f=\sum_{n\ge -v_P(A)} a_n t^n
		\in K(\!(t)\!). $$
		For $m\ge 0$ an integer, define
		\begin{align*}
			\ev_{A,t}(P,m):L(A)&\to K\\
			f&\mapsto a_{-v_P(A)+m}.
		\end{align*}
		Then $\ev_{A, t}(P, m)$ is a \textit{(linear) evaluation form} on $L(A)$.
	\end{definition}
	\begin{example}\label{ex:weight}
		Consider the rational
		function field $F=K(x)$ and take $A=MP_\infty$. Then
		$$L(A)=L(MP_\infty)=K[x]_{\le M}.$$
		The places of $F$ are given by the finite places $P_\alpha$, for
		$\alpha\in K$, together with the place at infinity $P_\infty$.
		Let $\alpha\in K$ and take the local parameter $t=x-\alpha$ at $P_\alpha$.
		Since $v_{P_\alpha}(A)=0$, every $f\in K[x]_{\le M}$ has a local expansion
		$$f(x)=a_0+a_1(x-\alpha)+\cdots+a_M(x-\alpha)^M.$$
		Hence, for $0\le m\le M$, the form $ \ev_{A,t}(P_\alpha,m):L(A)\to K$
		evaluates $\ev_{A,t}(P_\alpha,m) (f)=a_m$. This is precisely the $m$th Hasse derivative functional:
		\begin{align*}
			D^{(m)} :K[x]_{\le M}&\to K,\\
			f&\mapsto a_m.
		\end{align*}
		At the place at infinity, considering the local parameter $t=1/x$, we can write
		$$ f=\sum_{i=0}^M c_i x^i =\sum_{i=0}^M c_i t^{-i}=\sum_{n\ge -M} a_n t^n.$$
		Since $v_{P_\infty}(A)=M$, we have
		$\ev_{A,t}(P_\infty,m) (f) =a_{-M+m} $.
	\end{example}
	\begin{remark} \label{rem:hasse}
		In general it holds that 
		$$\ev_{A, t}(P, m)(f) = D^{(m)} \left( t^{v_P(A)} f \right)(P),$$
		where $D^{(m)}$ is the $m$th Hasse derivative. 
	\end{remark}
	Every linear form in $L(A)^\ast$ can be expressed as a sum of a finite number of evaluation forms, also see Proposition \ref{prop:linind} below. Hence, they form a generating set of $L(A)^\ast$, but not a basis, since not all such expressions will give unique linear forms. To be able to distinguish between several expressions that represent the same linear form we introduce the notion of formal expressions. 
	
	\begin{definition}
		Fix a local parameter $t_P$ for every place $P$ of $F$. Then elements in the space
		$$\bigoplus_{\substack{P\text{ a place} \\  m\geq 0}} \lambda_{P, m}\ev_{t_P}(P, m) \quad \text{ for } \lambda_{P, m} \in K$$
		are called \textit{formal expressions in evaluation forms} or just \textit{expressions}. 
	\end{definition}
	\begin{remark}
		We will usually denote expressions by $X, Y$ or $Z$ and linear forms by $\varphi, \chi$ or $\psi$. A priori an expression is not a linear form, because it is not specified on which space $L(A)$ it acts. For each divisor $A$ we can \textit{project} an expression $X$ to a linear form $X_A$ by sending $\ev_t(P, m)$ to $\ev_{A, t}(P, m)$ for all $P$ and $m$. We also say that \textit{$X$ represents $X_A$}.
		Note in particular that two distinct expression $X_1$ and $X_2$ may project to the same linear form on $L(A)$.  
	\end{remark}
	
	In the following we fix a choice of local parameters $t_P$ for each place $P$ and remove $t_P$ from the subscript of both $\ev_{A, t}$ and $\ev_t$. The subscript will only appear in some explicit calculations. 
	
	We are now ready to introduce the notion of weight of an expression. 
	
	\begin{definition}
		Let $X$ be an expression of evaluation forms of the form $$X = \sum_{P, m} \lambda_{P, m} \ev (P, m) \quad \text{ with } \lambda_{P, m} \in K.$$ 
		For each $P$ set 
		$$r_P(X) := \begin{cases} -1 &\text{ if } \lambda_{P, m} = 0 \text{ for all } m; \\
			\max\{m\mid \lambda_{P,m}\neq 0\} &\text{ else.}
		\end{cases}$$
		Define the \textit{divisor of $X$} as 
		$$D_X := \sum_{P} (r_P(X) + 1)P$$
		and the \textit{weight of $X$} as 
		$$W(X) := \deg(D_X).$$ 
		Finally, for a linear form $\varphi \in L(A)^\ast$ define \textit{the weight of $\varphi$} as 
		$$w_A(\varphi) := \min\{W(X) \mid X \text{ represents } \varphi \}.$$ 
	\end{definition}
	
	\begin{remarks}\label{re:exweight}~
		\begin{itemize}
			\item The weight
			$w_A(\varphi)$ for a linear form $\varphi$ is not necessarily attained by a unique expression $X$. For example, take $F=K(x)$ and $A = 8P_\infty$, and consider 
			$$X_1= \ev_{x}(P_0,4)  \quad \text{and}\quad X_2= \ev_{\frac 1x}(P_\infty,4).$$ 
			Then $X_1$ and $X_2$ represent the same linear form $\varphi$ in $L(A)^\ast$ and both have weight $5$. Moreover, by Lemma \ref{lem:weightLD} below we have  $w_A(\varphi) = 5$.
			\item The weight of a linear form does not depend on the choice of local parameters. For let $t$ and $t'$ be two local parameters at $P$. 
			Then for all integers $j$ there exist $b_{i, j} \in K$ with $b_{j, j} \neq 0$ such that
			$$(t')^j=\sum_{i\geq j}b_{i, j}t^i.$$
			For any $f\in L(A)$ we have 
			$$f=\sum_{j=-v_P(A)}^{\infty}a_j(t')^j,\qquad\text{ where }\qquad a_j=\operatorname{ev}_{t', A}(P,j+v_p(A))(f).$$
			Substituting $(t')^j$ into the expansion of $f$ gives
			$$f=\sum_{j=-v_P(A)}^{\infty}a_j\sum_{i\geq j}b_{i,j}t^i.$$
			Therefore, for every $m \geq -v_P(A)$ we have,
			$$\operatorname{ev}_{t,A}(P,m + v_P(A))=
			\sum_{j=-v_P(A)}^{m}
			b_{m,j}\operatorname{ev}_{t',A}(P,j+v_P(A)),$$
			which both have weight $m + v_P(A) + 1$, because $b_{m, m} \neq 0$.  
		\end{itemize}
	\end{remarks}
	
	\begin{definition}
		Let $E=\sum_{P} r_P P \ge 0 $ be a divisor. We define the following $K$-linear subspaces:
		\begin{align*}
			\divi E &:= \langle  \ev (P, m) \mid  0 \le  m\le r_P-1  \rangle;\\
			\divi {E}_A &:= \langle  \ev_{A} (P, m) \mid   0\le m\le r_P-1  \rangle.    
		\end{align*}
		Let $X$ be an expression of evaluation forms, the \textit{support of $X$} is defined as
		$$ \Supp (X):= \divi{D_X}.$$
		Let $\varphi\in L(A)^*$ be a linear form that is represented by a unique expression $X$ of weight $W(X)= w_A(\varphi)$.
		The \textit{support of $\varphi$} is defined as the $K$-linear subspace
		$$ \Supp (\varphi):= \divi{D_X}_A \subseteq L(A)^\ast.$$
	\end{definition}

	\begin{remark}
		The subspace $\divi{E}$ is generated by the formal expressions 
		$$\{ \ev(P,m)\mid 0\le m\le r_P-1 \}. $$
		These expressions are linearly independent over $K$, hence we have: 
		$$  \dim \divi E = \sum_{P} r_P   =\deg (E).  $$
		However, the set
		$$\{ \ev_A(P,m)\mid 0\le m\le r_P-1 \} $$
		is not necessarily linearly independent over $K$. Indeed, we have 
		$$\dim \divi{E} _A \le \deg E .$$
	\end{remark}
	
	\section{Linear forms and their representations}\label{sec:lems}
	In this section we develop the theory for evaluation forms, expression and weights further by proving some basic properties about them. The main result of the section is Proposition \ref{prop:m_dimension}, where we give a necessary condition for the annihilator of a square space to contain a form of high weight.
	
	We take $A$ to be any divisor such that $\deg(A) \geq 2g - 1$. 
	\subsection{Evaluation spaces, expressions and weight}
	We start by estimating the dimension of $\divi E_A$ from below.
	\begin{proposition} \label{prop:linind}
		Let $E \geq 0$ be a divisor and consider $\langle E \rangle_A \subseteq L(A)^\ast.$
		\begin{enumerate}[\normalfont (i)]
			\item Suppose that $\deg (E) \leq \deg (A) - 2g + 1$. Then, we have
			$$\dim \langle E \rangle_A = \deg(E).$$ 
			\item Suppose that $\deg(E) = \deg(A) - 2g + 1 + d$ for some $0 \le d \le 2g + 1$. Then, we have
			$$\dim \langle E \rangle_A \geq \deg(E) - (d+1)/2.$$ 
		\end{enumerate}
	\end{proposition}

	\begin{proof}
		Let $P_1, \ldots, P_\sigma$ be the places in the support of $E$ and $r_i \geq 0$ such that $$E = (r_1 + 1) P_1 + \ldots + (r_\sigma + 1) P_\sigma.$$ Define
		\begin{align*}
			\ev : L(A)&\to K^{\deg(E)},\\
			f&\mapsto \big(\ev_A(P_i,m)(f)\big)_{\substack{ 1 \leq i \leq \sigma \\ 0 \leq m \leq r_i}}.
		\end{align*}
		We first prove that $\ker(\ev)=L(A-E)$. 
		
		For $i = 1, \ldots, \sigma$ write $t_i$ for the local parameter $t_{P_i}$ at $P_i$ and let
		$f\in \ker(\ev)$, then for each $i$ we can write the local expansion at $P_i$ as
		$$f=\sum_{n\ge -v_{P_i}(A)} a_{i,n}t_i^n.$$
		The condition $f\in \ker(\ev)$ implies that
		$$a_{i,-v_{P_i}(A)}=a_{i,-v_{P_i}(A)+1}=\cdots=a_{i,-v_{P_i}(A)+r_i}=0.$$
		Hence,
		$$ v_{P_i}(f)\ge -v_{P_i}(A)+r_i+1. $$
		It therefore holds for all $P_i$ that
		$$v_{P_i}\big((f)+A-E\big)\ge 0.$$
		For a place $Q\notin\{P_1,\ldots,P_\sigma\}$, we have
		$v_Q(E)=0$. Since $f\in L(A)$, we find
		$$v_Q\big((f)+A-E\big)= v_Q\big((f)+A\big)\ge 0.$$
		We deduce that $f\in L(A-E)$ and thus, $\ker(\ev) \subseteq L(A-E)$.
		
		Conversely, let $f\in L(A-E)$. Then for each $i \in [1, \sigma]$ it holds that
		$$v_{P_i}(f)+v_{P_i}(A-E)\ge 0.$$
		Since $v_{P_i}(E)=r_i+1$, we get $v_{P_i}(f)\ge -v_{P_i}(A)+r_i+1$.
		Therefore $\ev(P_i,m)(f)=0$ for all $0\le m\le r_i$, and thus, $f\in\ker(\ev)$ and $L(A-E) \subseteq \ker(\ev).$ This proves that indeed $\ker(\ev) = L(A-E)$. 
		
		By the definition of $\ev$ and rank nullity it must now hold that 
		\begin{equation}\label{eq:imev} 
			\dim \divi E_A = \dim \operatorname{Im}(\ev) = \ell(A) - \ell(A-E).
		\end{equation}
		In case (i), we have that $\deg(A)\ge 2g-1+\deg(E)$. Therefore, Theorem \ref{thm:RR} gives
		$$\dim \divi {E}_A = \ell (A)-\ell (A-E)=\deg(E).$$
		In case (ii), we find $\deg(A-E) = 2g-1-d$ and can therefore apply Clifford's Theorem. This gives
		$$ \ell(A-E ) \le 1+\frac{1}{2}{\deg (A-E)}.$$
		Combining this with Equation \eqref{eq:imev} we find
		\begin{align*}
			\dim \divi{E}_A &\ge \frac{\deg(A) + \deg(E)}2 - g =\deg E- \frac{d+1}{2}. 
		\end{align*}
		This proves both statements. 
	\end{proof}
	The argument used in the proof of Proposition \ref{prop:linind} also gives a useful description of the annihilator of a Riemann--Roch space.
	\begin{lemma}\label{lem:rranh}
		Let $E$ be a divisor such that $A - E \geq 0$ and consider $L(E) \subseteq L(A)$. Then 
		$$\Phi_{A}(L(E)) = \langle A - E \rangle_A. $$ 
	\end{lemma}
	\begin{proof}
		We use the same ideas as in the proof of Proposition \ref{prop:linind}, but interchange the divisors $E$ and $A - E$. The proof implies that 
		$$\ev : L(A) \rightarrow K^{\deg(A-E)}$$ has kernel $L(E)$. By definition, this gives us $\langle A - E \rangle_A \subseteq \Phi_A(L(E))$. Furthermore, it also implies that 
		$$\dim \langle A - E \rangle_A = \ell(A) - \ell(E) = \dim \Phi_A(L(E)).$$
		Because all the subspaces are finite dimensional the statement follows.
	\end{proof}
		
	\begin{corollary}\label{cor:representative}
		Suppose $f \in U \subset L(A)$ such that $(f)_\infty = A$ and define $E := (f)_0$. Then, the following statements hold.
		\begin{enumerate}[\normalfont (i)]
			\item Every form in $\Phi_A(U)$ has a representative in $\divi {A}_A$ and one in $\divi {E}_{A}$.
			\item We have $\dim \divi A = \dim \divi A_A + g$ and $\dim \divi E = \dim \divi E_A + g$. 
		\end{enumerate}
	\end{corollary}
	\begin{proof}
		We assumed that $1, f \in U$, so
		\begin{align*}
			\Phi_{A}(U) &\subseteq \Phi_A(L(0)) = \langle D \rangle_A \qquad \text{ and } \\
			\Phi_{A}(U) &\subseteq \Phi_A(L(A-(f))) = \langle E \rangle_A,
		\end{align*}
		where the equalities follow from Lemma \ref{lem:rranh}. This proves (i). 
		
		For (ii) consider the divisor $A' = A + P$ for $P \not\in \Supp(A)\cup\Supp(E)$. Then the degree of $A'$ equals $\deg(A) + 1$ and hence, by Proposition \ref{prop:linind}, we have 
		$$\dim \divi {A'}_A \geq \deg(A) + 1 - g = \ell(A).$$ 
		This implies that $\divi {A'}_A = L(A)^\ast$. Because $f \in U$ we also have that $f \in L(A)$ and since $P$ is not a zero of $f$, we find $\ev_A(P, 0)(f) \neq 0$. Hence,
		$$\dim \divi A_A = \dim \divi {A'}_A  - 1= \deg A - g = \dim \divi A - g.$$ The statement for $E$ goes analogous. 
	\end{proof}
	
	Given a linear form $\varphi$ in $L(A)^\ast$ and $f \in F$ we can naturally define a linear form $\varphi'$ in $L(A-(f))^\ast$ that has kernel $f\Ker(\varphi)$ by sending $\varphi'(fh)$ to $\varphi(h)$. The following result tells us that the support of these forms can be assumed to be equal and thus, the support is in some sense translation invariant. 
	
	\begin{lemma}\label{lem:translate}
		Let $\varphi\in L(A)^*$ and $f\in F$. Define $A'= A-(f)$ and consider the linear form $\varphi'\in L(A') = fL(A)$ defined by
		\begin{equation*}
			\varphi': fh \mapsto \varphi(h).
		\end{equation*}
		Suppose $\varphi= X_A$ for some expression $X$, then there exists $X'$ such that $\Supp(X') = \Supp(X) $ and $\varphi'=X' _{A'}$.
	\end{lemma}
	\begin{proof}
		Let $E \geq 0$ be a divisor and let $\psi \in \divi E_A$. Define $\psi' \in L(A')^\ast$ analogously to $\varphi'$, that is,
		$$\psi'(fh)=\psi(h).$$
		
		As the first step, we show that the map sending $\psi\in\divi E_A$ to $\psi'\in\divi E_{A'}$ is a bijection. By Lemma \ref{lem:rranh} it holds that $\psi \in \Phi_A(L(A-E))$ which is equivalent to $L(A-E) \subseteq \Ker(\psi)$. Multiplying both sides by $f$ we find that 
		$$L(A' - E) = fL(A - E) \subseteq f\Ker(\psi) = \Ker(\psi').$$ 
		Reversing the same argument, we find that indeed $\psi' \in \divi E_{A'}$. Since $\psi(f^{-1}(fh)) = \psi'(fh)$, we can also get $\psi$ back from $\psi'$, showing that the map sending $\psi$ to $\psi'$ must be a bijection. From this we additionally conclude that $\dim \divi E_A = \dim \divi E_{A'}$. 
		
		As our second step, we prove the claim that if there exists an expression $Y \in \divi E$ such that $D_Y = E$ and $Y_A = 0$, there also exists an expression $Y' \in \divi E$ such that $D_{Y'} = E$ and $Y'_{A'} = 0$.
		Assume this is not the case; that is there is no such $Y'$. Then there must be some $P \in \Supp E$ for which there exists no expression $Z_P \in \divi E$ such that 
		\begin{align}
			Z_{P, A'} = 0 \qquad\text{and}\qquad v_P(D_{Z_P}) = v_P(E), \label{eqn:ZP}
		\end{align}
		indeed, if such a $Z_P$ would exist for all $P \in \Supp(E)$ we can construct $Y'$ by taking a linear combination 
		$$Y':= \sum_P \lambda_{P }Z_P,$$
		contradicting the assumption $D_{Y'} \neq E$. Hence, fix $P$ to be a place for which no $Z_P$ exists satisfying \eqref{eqn:ZP}. Then every element in the kernel of $\divi E\to \divi E_{A'}$ lies in $\divi{E-P}$. Since $\divi{E-P}$ has co-dimension one in $\divi E$, rank-nullity gives $$\dim \divi {E-P}_{A'}=\dim \divi E_{A'} - 1.$$ However, since  $v_P(D_Y) = v_P(E)$ and $Y_D = 0$, we have $\dim \divi {E-P}_A = \dim \divi E_A$. Using also the first step we find
		$$\dim \divi E_{A'} = \dim \divi E_A = \dim \divi {E - P}_A = \dim \divi {E-P}_{A'} = \dim \divi E_{A'} - 1,$$
		a contradiction. This concludes the proof of the claim. 
		
		Now we are ready to consider $\varphi$. By the first step we immediately get that $\varphi' \in \divi {D_X}_{A'}$ and thus is represented by some $X'$ such that $\Supp(X') \subseteq \Supp(X)$. It remains to show that $X'$ can be chosen such that the two supports are equal. Since we have the bijection from $\divi {D_{X'}}_A$ to $\divi {D_{X'}}_{A'}$, we can also find $X'' \in \Supp(X')$ such that $X''_{A} = \varphi = X_A.$ Hence, $(X - X'')_A = 0$. Take $E = D_{X - X''}$, then by the second step we know that there exists some $Y \in \divi E$ such that $D_Y = E$ and $Y_{A'} = 0$. Therefore, for some $\lambda \in K$ we must have $\Supp(X' + \lambda Y) = \Supp(X)$ and $(X' + \lambda Y)_A = \varphi'$. 
	\end{proof}
	
	\begin{remark}
		Lemma \ref{lem:translate} implies that the dimension of $\divi E_A$ does not depend on the exact divisor $A$, but only on the class of the Picard group that $A$ belongs to. 
	\end{remark}
	
	By Lemma \ref{lem:translate} the following is now well-defined.  
	
	\begin{definition}
		For $\varphi \in L(A)^\ast$ and $f \in F$ we define the \textit{translate of $\varphi$ by $f$} as the linear form $\varphi' \in L(A - (f))^\ast$ such that 
		$$\varphi': fh \mapsto \varphi(h).$$ 
		Furthermore, let $X$ be an expression of evaluation forms representing $\varphi$. A \textit{translated expression of $X$ from $\varphi$ to $\varphi'$} is defined as an expression of evaluation forms $X'$ such that $\Supp(X) = \Supp(X')$ and $X'_{A - (f)} = \varphi'$. By slight abuse of notation we will also call $X'$ the \textit{translate of $X$ by $f$}. 
	\end{definition}
	
	As mentioned in Section \ref{Sec:defs}, it is important to make a distinction between expressions and linear forms, because there might be multiple expressions that represent the same linear form. The following result shows that if the weight of  a linear form is sufficiently small, then there is in some sense a canonical expression to choose. This upper bound does not hold for all linear forms, also see Remark \ref{re:exweight}, but in Section \ref{Sec:specific} we will see that it does always hold for linear forms in $\Phi_{2D}(S^2)$. 
	
	\begin{lemma}\label{lem:weightLD}
		Let $\varphi \in L(A)^\ast$ and let $X_1$ and $X_2$ be two distinct expressions that represent $\varphi$. Then 
		$$W(X_1) + W(X_2) \geq \deg A - 2g + 2.$$ 
		In particular, if $$W(X_1) \leq \frac{\deg A - 2g + 1}{2}$$ 
		it is the unique expression representing $\varphi$ such that $W(X_1) = w_A(\varphi)$. 
	\end{lemma}
	\begin{proof}
		Define $E := D_{X_1} + D_{X_2}$, then 
		$$ 0 = (X_1 - X_2)_A \in \divi E_A,$$
		where $X_1 - X_2$ is a non-zero expression. Proposition \ref{prop:linind} implies that $\deg(E) \geq \deg A - 2g + 2$ and thus, 
		$$\deg A - 2g + 2 \leq \deg E \leq W(X_1) + W(X_2).$$
		
		The uniqueness statement follows immediately. 
	\end{proof}
	\subsection{Vanishing on square spaces}
	In the additive case, it is an easy fact that if $m \not\in B+B$, then $|B \cap [0, m]| \leq (m+1)/2$, since for each pair $\{a, m-a\} \subseteq [0, m]$ at most one element can be in $B$. Proposition \ref{prop:m_dimension} and Corollary \ref{cor:m_dimension} generalize this statement to the linear setting, replacing $m$ by the weight of a linear form in $\Phi_{2A}(U^2)$. 
	This shows that the notion of weight really encapsulates the complexity of a linear form.
	
	\begin{proposition}\label{prop:m_dimension}
		Suppose that $U \subseteq L(A)$ is a $K$-subspace, so that
		$U^2 \subseteq L(2A)$.
		Let $\varphi \in \Phi_{2A }(U^2)$ be a form and $X$ an expression of weight $w$ such that $\varphi= X_{2A}$.
		Define
		$$T:=U\cap\bigcap_{\chi\in \divi{D_X} _A}\ker(\chi).$$
		Then we have
		$$\operatorname{codim}_U(T) \le \frac{w}{2}.$$
	\end{proposition}
	\begin{proof}
		When $\dim (U) \leq \frac w 2$ the statement is trivial, so we assume that $\dim U > \frac w2$. The rest of the lemma we prove by contradiction, so suppose that $\codim_U(T) = x > w/2$. 
		
		Write $\operatorname{Gen}_A(X)$ for the set of generating evaluation forms $\ev_A(P, j)$ of $\divi {D_X}_A$, that is
		$$ \Gen_{A}(X)= \{\ev_{A}(P_i,j):1\le i\le \sigma,\ 0\le j\le r_i\} $$
		where $r_i= r_{P_i} (X)$. We can then equivalently define $T = U \cap \bigcap_{\chi \in \Gen_{A}(X)} \ker(\chi).$
		
		We order $\Gen_A(X)$ lexicographically: first by the index $i$ of the places and then by the index $j$, both in increasing order. This order is denoted by $\leq$. For example, we have $\ev_A(P_1, 1) \leq \ev(P_1, 3) \leq \ev(P_2, 0).$ 
		
		Choose elements 
		$$\psi_1, \ldots , \psi_x \in \Gen_A(X)  $$
		inductively as follows
		\begin{enumerate}[1.]
			\item Set $U_0:=U$.
			\item For each $1\le \ell \le x$, let $\psi_\ell$ be the smallest element of $\Gen_A(X)$ with respect to $\leq$ such that
			$$
			U_\ell:=U \cap\bigcap_{\chi\in\Gen_A(X), \chi\le \psi_\ell }\ker(\chi)
			\subsetneq U_{\ell-1}
			$$
			and $\operatorname{codim}_U(U_\ell )=\ell$.
		\end{enumerate}
		For $\ell = 1, \ldots, x$, define $i_\ell$ and $j_\ell$ such that
		$$\psi_{\ell}=\ev_{A}(P_{i_\ell},j_\ell).$$
		Then $T = U \cap \bigcap_{\ell = 1}^x \ker(\psi_\ell)$.
		We choose $u_\ell \in U_{\ell-1} \setminus U_\ell $ so that 
		\begin{equation} \label{eq:ul}
			\ev_A(P_i,j)(u_\ell) =\begin{cases}
				1&\text{if } i=i_\ell \qquad \text{and} \qquad j=j_\ell;\\
				0&\text{if } i= i_\ell \qquad\text{and}\qquad  j< j_\ell;\\
				0&\text{if } i< i_\ell.
			\end{cases}
		\end{equation} 
		Then in particular we have that
		\begin{align}
			v_{P_{i_\ell} } (u_\ell) =-v_{P_{i_\ell}} (A) +j_\ell .
			\label{eq:s_u}
		\end{align}
		Now, we define 
		\begin{alignat*}{2}
			& \Gamma_1 := \{\psi_1, \ldots, \psi_x\}, &&\\
			&\Gamma_2:=\{\ev_A( P, j) \in \Gen_A(X) \mid \ev_A (P,r_P(X) - j)\in \Gamma_1\}, &&\\
			& B:= \Gen_A(X) \setminus \Gamma_2 \quad \text{ and }&&\\
			& V := U \cap \bigcap_{\chi\in B}\ker(\chi). &&
		\end{alignat*}
		Then $|B|=w-x<x $ and thus, $\dim V >\dim T.$ Let $u_0 \in V\setminus T$ and note that $T = V \cap \bigcap_{\chi \in \Gamma_2} \ker(\chi)$. Because $u_0\notin T$, there must be some element $\psi_0 \in \Gamma_2$ such that $\psi_0(u_0) \neq 0$. Hence, we can take $i_0$ maximal such that for some $j$ we have that
		$$\ev_A(P_{i_0},j )\in \Gamma_2 \qquad\text{and}\qquad \ev_A(P_{i_0},j)(u_0)\neq 0.$$
		We define $j_0$ to be the minimal $j$ for which this holds. Then, in particular, we have 
		\begin{align}
			v_{P_{ i_0} } (u_0)  = -v_{P_{i_0}}(A) +j_0    \label{eq:d}.
		\end{align}
		Note that by the construction of $\Gamma_2$, there exists $\ell$ such that 
		\begin{equation}\label{eq:evpsi} \psi_\ell=\ev_{A}(P_{i_0},r_{i_0}-j_0) \in \Gamma_1.
		\end{equation}
		From now on we fix $\ell$ such that \eqref{eq:evpsi} holds, that is $i_\ell = i_0$ and $j_\ell = r_{i_0} - j_0$. 
		
		After possibly normalizing $u_0$, we claim that for $i \in [1, \sigma]$ and $j \in [0, r_i]$ we have
		$$ \ev_{ 2A}(P_i,j) (u_0 u_\ell) = \begin{cases}
			1 &\quad \text{if}\quad i=i_0 \quad\text{and}\quad j=r_{i_0};\\0& \quad\text{otherwise}.
		\end{cases}$$
		
		When $i \neq i_0$ we can use the maximality of $i_0$ and Equation \eqref{eq:ul} to find
		\begin{align*}\ev_A(P_i,j) (u_0) =0 
			&\quad\text{if}\qquad i> i_0 \qquad \text{ and } \qquad  j\le r_i,\\
			\ev_A(P_i,j) (u_\ell) =0 &\quad\text{if}\qquad i< i_0 \qquad \text{ and } \qquad j\le r_i.
		\end{align*}
		This implies that $v_{P_i}(u_0 u_\ell) \geq -v_{P_i}(2A) + r_i + 1$ for these $i$ and thus, the claim for $i \neq i_0$ must hold.
		
		When $i = i_0$, by Equations \eqref{eq:s_u} and \eqref{eq:d}, 
		we have 
		$$ v_{P_{i_0}}(u_0 u_\ell) = -v_{P_{i_0}}(2A)+ r_{i_0}.  $$ 
		This directly implies the claim for $j < r_{i_0}$ and because we are allowed to normalize $u_0$ we also get the result for $j = r_{i_0}$. 
		
		Finally, recall that $\varphi$ has the following decomposition
		$$\varphi = \lambda \ev_{2A} (P_{i_0} ,  r_{i_0}) + \sum_{ (i,j)\neq (i_0,r_{i_0})}\lambda_{i j} \ev_{2A}(P_i,j ), $$
		for some non-zero $\lambda$. Hence, evaluating $\varphi$ at
		$u_0 u_\ell \in U^2$ we get
		$$\varphi (u_0 u_\ell) = \lambda + 0\neq 0. $$
		This contradicts the fact that $\varphi\in \Phi_{2A}(U^2).$ We conclude that $x \leq \frac w2$, as wished.
	\end{proof}
	
	\begin{remark} \label{rem:reform}
		The following equivalent formulation of Proposition \ref{prop:m_dimension} will be useful in the proof of Corollary \ref{cor:m_dimension} below. Consider the restriction map
		\begin{align*}
			\rho_X^U:\divi{D_X}_A &\longrightarrow U^*, \\
			\psi &\longmapsto \psi|_U .
		\end{align*}
		Let $T \subseteq U$ be the space that is annihilated by all forms in $\operatorname{Im}(\rho_X^U)$, that is
		$$T= U\cap\bigcap_{\psi\in\divi{D_X}_A }\ker(\psi). $$
		We get
		$$ \dim U^\ast = \dim U= \dim \left(\operatorname{Im}(\rho_X^U) \right)+ \dim T,  $$
		and hence we obtain
		$$\dim \left(\operatorname{Im} (\rho_X^U ) \right) = \operatorname{codim}_U (T). $$
		Thus Proposition \ref{prop:m_dimension} bounds the number of independent linear conditions on $U$ imposed by $\divi{D_X}_A$, that is
		$$\dim \left(\operatorname{Im}(\rho_X^U)\right) \leq \frac{W(X)}{2}. $$
	\end{remark}
	
	\begin{corollary}\label{cor:m_dimension}
		Let $X$ be an expression of weight $W(X) = w$ such that $X_{2A} \in \Phi_{2A}(U^2)$.
		\begin{enumerate}[\normalfont (i)]
			\item Suppose that $w \leq  \deg(A) -2g +1 $. Then we have
			$$ \dim (\divi{D_X}_A  \cap \Phi_{A} (U) )\ge w/2. $$
			\item Suppose that  $w = \deg(A) -2g +1+d $ for some $d \in \{0,\ldots, 2g +1\}$.  Then we have
			$$ \dim( \divi{D_X}_A  \cap \Phi_{A} (U))\ge \frac{w}{2}-\frac{d+1}{2}.$$ 
		\end{enumerate}
	\end{corollary}
	\begin{proof}
		Let $\rho_X^U$ be defined as Remark \ref{rem:reform}. Note that 
		$$\ker (\rho_X^U)=  \divi{D_X}_A  \cap \Phi_A (U) .$$
		By rank-nullity, we get
		\begin{align*}
			\dim\bigl(\divi{D_X}_A\cap \Phi_A(U)\bigr)
			= \dim \divi{D_X}_A-\dim \operatorname{Im}(\rho_X^U)  \geq \dim \divi{D_X}_A-\frac{w}{2}.
		\end{align*}
		
		If $w\le \deg(A) -2g+1 $, by Proposition \ref{prop:linind}(i), we have
		$$\dim \divi{D_X}_A =\deg ({D_X} )=w. $$
		Thus, we get
		$$ \dim\bigl(\divi{D_X}_A\cap \Phi_A(U)\bigr)\ge w/2$$
		as desired. This concludes the first part. 
		
		If $ w =  \deg(A) -2g +1+d $, by Proposition \ref{prop:linind}(ii), we have
		$$\dim \divi{D_X}_A \ge \deg (D_X ) -\frac{d+1}{2} =w -\frac{d+1}{2}.$$
		This concludes the second part, as we have
		\begin{equation*} \dim\bigl(\divi{D_X}_A\cap \Phi_A(U)\bigr)\ge \frac w2 -\frac{d+1}{2}. \qedhere \end{equation*}
	\end{proof}
	
	The result has the following consequence. 
	\begin{corollary} \label{cor:weight2}
		Let $U \subseteq F$ be a finite dimensional $K$-subspace, such that $A$ is the minimal divisor for which $U \subseteq L(A)$. Then any $\varphi \in \Phi_{2A}(U^2)$ has weight at least $2$. 
	\end{corollary}
	\begin{proof}
		Suppose that there is some $\varphi$ of weight $1$ in $\Phi_{2A}(U^2)$. Then $\varphi = \lambda \ev(P, 0)$ for some place $P$ and $\lambda \in K^\times$. By Corollary \ref{cor:m_dimension} we find that $\ev(P, 0) \in \Phi_{A}(U)$. This implies that $U \subseteq L(A-P)$, contradicting the minimality of $A$. 
	\end{proof}
	
	\subsection{Bounded weight}
	The goal of this subsection is to prove Proposition \ref{prop:max_weight}, which gives an upper bound on the weight $w_A(\varphi)$ of a form in $L(A)^\ast$. The method used to prove this is quite different from the results we saw before and we will need some notions from algebraic geometry about secant varieties. Furthermore, in this subsection (and only this subsection) some of the notation deviates from the notation in the rest of the paper.
	
	\begin{definition}
		Let $X \subseteq \mathbb{P}^M$ be an embedded projective variety over $K$ and $k \in \Z_{\geq 0}$. Define 
		$$S_k^0(X) := \{ u \in \langle P_0, \ldots, P_k \rangle \mid P_i \in X, \text{ distinct} \}.$$ 
		Then the Zariski closure of $S_k^0(X)$ is the \textit{$k$th-secant variety of $X$} denoted by $S_k(X)$. 
	\end{definition}
	
	Using this definition we have the following result by Lange for curves.
	
	\begin{theorem}[{\cite{Lange1984Secants}}]\label{Thm:secantk}
		Let $X \subseteq \mathbb P^M$ be an irreducible non-degenerate curve. Then the variety $S_k(X)$ has dimension $\min(2k + 1, M).$ 
	\end{theorem} 
	
	There exists an alternative equivalent definition (see for example \cite[Propostion 6.8]{Buczynski2017FinSchemesSecantVar}) of $k$th-secant varieties using finite subschemes. The advantage of this definition is that we do not have to take a closure and it is therefore immediately clear what the variety consists of. We first recall the definition of a finite subscheme and then explain how to interpret the other definition when $X$ is a smooth curve. 
	
	\begin{definition}
		Let $X$ be a variety over $K$. A \textit{finite subscheme} $H$ of $X$ is a zero-dimensional subscheme of finite length. The scheme parameterizing all finite subschemes of length $k \in \Z_{\geq0}$ is called the \textit{Hilbert scheme of $k$ points} denoted by $\text{Hilb}^{k}(X)$. 
	\end{definition}
	
	\begin{definition}
		Let $X \subseteq \mathbb P^M$ be an embedded projective variety over $K$ and $k \in \Z_{\geq 0}$. The \textit{$k$th-secant variety of $X$} is defined as  
		$$S_k(X) := \bigcup_{H \in \operatorname{Hilb}^{k+1}(X)} \langle H \rangle,$$
		where $\langle H \rangle$ is the smallest projective linear space in $\mathbb P^M$ containing $H$.
	\end{definition}
	
	On a smooth curve $X$, a finite subscheme $H$ of length $k$ is essentially a collection of $k$ not necessarily distinct points together with infinitesimal information at these points. More precisely, suppose $X_{\text{aff}}$ is a suitable affine chart containing $H$ as a closed subscheme. Write $\Gamma(X_{\text{aff}})$ for the coordinate ring of $X_{\text{aff}}$, then this infinitesimal information is encoded in the coordinate ring $R_H$ of $H$ and the surjective map $R_H \rightarrow \Gamma(X_{\text{aff}})$.
	
	In particular, for $H = \{m_1 P_1, \ldots, m_\sigma P_\sigma\}$ we have $$R_H = \Gamma(X_{\text{aff}}) / I_H \qquad \text{ where } \qquad I_H = \mathfrak m_{P_1}^{m_1} \cap \ldots \cap \mathfrak m_{P_\sigma}^{m_\sigma}.$$ 
	Fix $i$ and write $P = P_i, m = m_i$, then locally at the point $P$ we have
	$$ I_{H, P} = \mathfrak m_P^{m} = (t_P^m),$$
	where $t_P$ is the local parameter at $P$. We will use this to deduce an explicit expression for $\langle H \rangle.$
	
	When $X$ is embedded in projective space $\mathbb P^M$ the space $\langle H \rangle$ equals the following
	$$\langle H \rangle = \bigcap_{\substack{L \text{ a hyperplane} \\ H \subseteq L \subseteq \mathbb P^M}} L.$$ 
	Let $L$ be a hyperplane in $\mathbb P^M$ then it is described by a linear form $h_L$. When we restrict $h_L$ to $X_{\text{aff}}$ the coordinate ring of $L\cap X_{\text{aff}}$ is $\Gamma(X_{\text{aff}})/(h_L)$. It now holds that $H \subseteq L$ if and only if $h_L \in I_H$. We can translate this to a collection of local statements, namely 
	$$H \subseteq L \quad \textup{ if and only if } \quad h_{L} \in (t_P^m) \text{ for all $P$ occurring in $H$.}$$
	Note that $h_{L} \in (t_P^m)$ is equivalent to saying $h_L(P), D^{(1)}h_L(P) \ldots, D^{(m-1)}h_L(P) = 0$, where $D^{(i)}$ denotes the $i$th Hasse derivative. 
	
	Suppose the embedding of $X$ in projective space is given by $P \mapsto (f_0(P): f_1(P) : \ldots : f_M(P))$ for some functions $f_i$. Define $$P^{(i)} := (D^{(i)} f_0(P) : \ldots : D^{(i)}f_M(P)).$$ 
	For any linear form $h$ we have $D^{(i)} h (P) = h(P^{(i)})$. Hence, we find that $H \subseteq L$ if and only if $$P_1^{(0)}, \ldots, P_1^{(m_1-1)}, \ldots, P_\sigma^{(0)}, \ldots, P_\sigma^{(m_\sigma-1)} \in L.$$ We conclude that 
	$$\langle H \rangle = \langle P_1^{(0)}, \ldots, P_1^{(m_1-1)}, \ldots, P_\sigma^{(0)}, \ldots, P_\sigma^{(m_\sigma-1)} \rangle.$$
	
	We are now ready to prove the main result of this section. 
	\begin{proposition}\label{prop:max_weight}
		Suppose that $\deg(A) \geq 2g + 1$ and let $\varphi \in L(A)^\ast$. The weight of $\varphi$ is at most 
		$$ \left \lceil \frac{ \deg A - g + 1 }2\right \rceil.$$ 
	\end{proposition}
	\begin{proof}
		Write $C$ for the curve associated to the function field $F$ and $M+1$ for the dimension of $L(A)$, thus $M = \deg A - g$. Let $f_0, \ldots, f_{M}$ be a basis for $L(A)$ and define $X$ as the projective model of $C$ associated to $A$, that is as the image of 
		\begin{align*}
			\Psi: C &\rightarrow \mathbb P^M \\
			P &\mapsto ((t_P^{v_P(A)}f_0)(P) : (t_P^{v_P(A)}f_1)(P) : \ldots : (t_P^{v_P(A)}f_M)(P)).
		\end{align*}
		Because $\deg A \geq 2g+1$, the divisor $A$ is very ample and the map $\Psi$ gives an embedding. In particular $X$ is a smooth projective curve embedded in $\mathbb P^M.$
		
		For a linear form $\psi$, let $\lambda_0, \ldots, \lambda_M \in K$ such that  
		$$\lambda_i := \psi(f_i) \qquad\text{and define }\qquad Q_\psi := (\lambda_0: \ldots: \lambda_M) \in \mathbb P^M. $$ 
		Then, using Remark \ref{rem:hasse}, we have $Q_{\ev(P, m)} = P^{(m)}$ defined above. This definition also admits a notion of linearity, in the sense that for $\chi$ and $\psi$ two linear forms it holds that 
		$$Q_{\chi} + Q_{\psi} = Q_{\chi + \psi}.$$ 
		
		By Theorem \ref{Thm:secantk} we know that $S_k(X) = \mathbb P^M$ whenever $k \geq \frac 12 (M-1)$. Fixing $k := \lceil \frac 12 (M-1) \rceil$ we find that $Q_\varphi \in S_k(X)$. Hence, we can find some finite subscheme $H$ of $X$ of length $k$ such that
		$$ Q_\varphi \in \langle H \rangle = \langle P_1^{(0)}, \ldots, P_1^{(m_1-1)}, P_2^{(0)}, \ldots, P_\sigma^{(m_\sigma-1)} \rangle.$$
		This implies that there must exist $\mu_{i, j} \in K$ such that 
		$$Q_\varphi = \mu_{0, 0}P_0^{(0)} + \ldots + \mu_{\sigma, m_\sigma-1}P_\sigma^{(m_\sigma - 1)}$$ 
		and therefore, 
		$$\varphi = \mu_{0, 0} \ev_A(P_0, 0) + \ldots + \mu_{\sigma, m_\sigma-1}\ev_A(P_\sigma, m_\sigma - 1),$$
		which has weight at most $k+1 = \lceil \frac 12 (M+1) \rceil$. 
	\end{proof}
	
	\section{Specific results for $S$} \label{Sec:specific}
	In this section we prove some results that depend on the specific structure of $S$, that is $\gamma \leq k-3$. By a fairly straightforward argument we find that the conjecture holds for $2b < k$. In the last subsection we will give some technical results that we will need in Section \ref{sec:main} to show that the conjecture also holds for $2b \leq k+2$.
	
	We recall the following intermediate result from \cite{couvreur2024freimanproof}.
	\begin{proposition}\label{prop:KfS2}
		Let $D$ be the minimal divisor such that $S \subseteq L(D)$, then for any $f \in S$ such that $(f)_\infty = D$ it holds that $K(f)S^2 = F$. 
	\end{proposition}
	
	\begin{remarks}~
		\begin{itemize}
			\item In \cite{couvreur2024freimanproof} this is shown for $f$ such that $F/K(f)$ is separable and all zeroes of $f$ are simple. However, one can always choose $a \in K$ such that these conditions hold for $f + a$ and since $K(f+a) = K(f)$ the more general conclusion follows. 
			\item Lemma \ref{prop:KfS2} can be seen as the linear version of the following additive result: for $B \subseteq [0, M]$ with $0, M \in B$ we have $B+B \mod M = \Z/M\Z$. 
		\end{itemize}
	\end{remarks}
	
	\subsection{The additive case} 
	We give a brief overview of Freiman's proof to be able to motivate some of the intermediate steps of the proof in the linear case. The reader may choose to skip this subsection, as all proofs below can be understood without the motivation. 
	
	Let $B \subseteq [0, M]$ be a finite set such that $0, M \in B$ and assume that $\gcd (B) = 1$. Then $B+B \subseteq [0, 2M]$. We define a \textit{hole in $B$} as any element $a \in [0, M]$ that is not in $B$ and similarly we define holes in $B+B$ as $a \in [0, 2M]$ not in $B+B$. Then $B$ has $b_B := M + 1 - |B|$ holes and $B+B$ has $2b_B - \gamma_B$ holes, where $\gamma_B := |B+B| - 2|B| + 1$. 
	Freiman proves that each hole of $B$ falls in exactly one of the following categories: 
	\begin{enumerate}[\normalfont (i)]
		\item A hole $a$ in $B$ is called a \textit{left stable hole} if $a$ is also a hole in $B+B$.
		\item A hole $a$ in $B$ is called a \textit{right stable hole} if $a + M$ is a hole in $B+B$. 
		\item A hole $a$ in $B$ is called an \textit{unstable hole} if neither $a$ nor $a + M$ is a hole in $B+B$. 
	\end{enumerate}
	In particular he shows that a hole cannot be left and right stable simultaneously and that all holes in $B+B$ come from either a left or a right stable hole in $B$. Therefore, we have exactly $2b_B - \gamma_B$ stable holes in $B$ and $\gamma_B - b_B$ unstable holes in $B$. 
	
	By a simple counting argument we know that if $a \in [0, M]$ or $2M - a \in [M, 2M]$ is a hole in $B+B$, then there must be at least $(a+1)/2$ holes in $B$ in the interval $[0, a]$ or $[M-a, M]$, respectively. This gives restrictions on how large the largest left stable hole $e$ and how small the smallest right stable hole $c$ in $B$ can be. Note that by definition it holds that $[e+1, c + M -1] \subseteq B+B$. Hence, the bounds on $e$ and $c$ give a bound on the minimal size of an arithmetic progression in $B+B$, which is exactly what we want.
	
	The most tricky part of the argument is to show that $e < c$, of which we do not give the details here. 
	
	The analog of the fact that a hole in $B$ can never be both left and right stable is given in Theorem \ref{lem:dimt1}. Lemma \ref{lem:disj} plays the role of the fact that all holes in $B+B$ arise from either left or right stable holes in $B$. The proof given in Section \ref{sec:main} and especially the proof of Proposition \ref{prop:main} is strongly based on the additive argument of $e<c$.
	\subsection{The case $2b < k$} 
	The main step of the proof when $2b < k$, is to give a better upper bound on the weight of a form in $\Phi_{2D}(S^2)$. To do this we will use that $b+g \leq \gamma \leq k-3$, which is implied by the following. 
	\begin{lemma}\label{lem:gammabound}
		We have $\gamma \in [b + g, 2b + g]$.
	\end{lemma}
	\begin{proof}
		The upper bound is immediate, because
		$$\gamma = \dim(S^2) - 2\dim(S) + 1 \leq \ell(2D) - 2(\ell(D) - b) + 1 = 2b + g.$$ 
		
		For the lower bound we need the following fact. Let $U, V$ and $T$ be $K$-subspaces such that $U = V \oplus T$ and let $L$ be an extension field of $K$. Then $\dim_L U \leq \dim_L V + \dim_L T$ and therefore,
		$$\dim_K U - \dim_L U \geq \dim_K V + \dim_K T - \dim_L V - \dim_L T \geq \dim_K V - \dim_L V.$$ 
		We will use this for $L = K(f)$, $U = S^2$ and $V = S + fS$. 
		
		Let $f \in S$ such that $(f)_\infty = D$. Then clearly $S + fS \subseteq S^2$, so
		\begin{equation*} \label{eq:dims}
			\dim_K S^2 - \dim_{K(f)} S^2 \geq \dim_K (S + fS) - \dim_{K(f)} (S+ fS).
		\end{equation*} 
		We know that $\dim_K S^2 = 2k - 1 + \gamma$ and $\dim_{K(f)} (S + fS) = k - 1$. Since $S \cap fS = \langle f \rangle$, it holds that $\dim_K (S + fS) = 2k-1$.  By Proposition \ref{prop:KfS2} we have that $K(f)S^2 = F$, which gives $\dim_{K(f)} S^2 = N$. Filling in these values we find that
		$$ 2k - 1 + \gamma - N \geq 2k - 1 - (k-1).$$ 
		Recall that $N = k + b + g - 1$, which gives us the lower bound. 
	\end{proof}
	
	\begin{lemma}\label{lem:2bbound}
		Let $\varphi\in \Phi_{2D}(S^2)$ then,
		$ w_{2D}(\varphi)\leq 2b. $
	\end{lemma}
	\begin{proof}
		From Proposition \ref{prop:max_weight} we know that for all $\varphi$ it holds that $w:= w_{2D}(\varphi) \leq N - g/2 + 1$. We first prove the bound for $w \leq N - 2g$ and then show that if $w$ is in the interval $[N- 2g + 1, N - g/2 + 1]$ we end up in a contradiction. 
		
		For $w \leq N - 2g$, by Corollary \ref{cor:m_dimension} (i), we know that there exist linearly independent forms $\varphi_1, \ldots, \varphi_{w/2} \in \Phi_D(S)$. Because the dimension of $\Phi_D(S)$ is $b$, we find that $w/2 \leq b$, which is what we wanted. 
		
		Now suppose that $w = N - 2g + 1 + d$ for some $d$ between $0$ and $2g + 1$. By Corollary \ref{cor:m_dimension} (ii) we find linearly independent
		$$\varphi_1, \ldots, \varphi_{\frac{w -( d + 1)}2} \in \Phi_D(S).$$ Hence, we obtain
		$$w \leq 2b + d + 1.$$ 
		From Lemma \ref{lem:gammabound} we know that $b \leq \gamma - g \leq k - 3 - g$. Substituting this into the previous inequality gives
		$$w \leq b + k - 3 - g + d + 1 = N - 2g + d - 1.$$ 
		This contradicts our earlier assumption and we conclude that $w_{2D}(\varphi) \leq 2b$. 
	\end{proof}
	
	\begin{remark}
		Because $2b \leq (k - g - 3) + b = N - 2g - 2$, it follows from Lemmas \ref{lem:weightLD} and \ref{lem:2bbound} that any $\varphi \in \Phi_{2D}(S^2)$ has a \textit{unique} minimal expression. We will also call this \textit{the} minimal expression of $\varphi$. 
	\end{remark}
	
	The following lemma gives a general lower bound of $2k$ on the weight of non-minimal expressions representing forms in $\Phi_{2D}(S^2)$. We prove it in a slightly more general setting, which will be useful for the proof of the main theorem later on.
	
	\begin{lemma}\label{lem:2kbound}
		Let $\varphi \in \Phi_{2D}(S^2)$ of weight at most $2b - 2z$ for some $z\geq 0$.  Let $X$ be an expression representing $\varphi$ that is not minimal. Then $W(X) \geq 2k + 2z.$
	\end{lemma}
	\begin{proof}
		Let $Y$ be the minimal expression of $\varphi$, then $$W(Y) = w_{2D}(\varphi) \leq 2b - 2z.$$ The expressions $X$ and $Y$ have to be distinct, so Lemma \ref{lem:weightLD} implies that 
		\begin{equation*}
			W(X) \geq 2N - 2g + 2 - W(Y) \geq 2k + 2z. \qedhere
		\end{equation*}
	\end{proof}
	
	\begin{proposition} \label{prop:zbound}
		Suppose all forms in $\Phi_{2D}(S^2)$ have weight at most $2b - 2z$. Then the conjecture holds whenever $2b \leq k + 3z-1.$ 
	\end{proposition}
	\begin{proof}
		It suffices to prove that 
		$$\Phi_{2D}(S^2) \subseteq \Phi_{2D}(L(G)),$$
		for some $L(G)$ of dimension at least $2k + g - 1.$ This implies that $\Phi_{2D}(L(G))$ should have dimension at most $2b$. By Lemma \ref{lem:rranh} this is equivalent to showing there exists a divisor $2D - G \geq 0$ such that $\Phi_{2D}(S^2) \subseteq \divi {2D - G}_{2D}$ and $\deg(2D - G) \leq 2b$. 
		
		Suppose this is not the case. Then there are two forms $\varphi, \psi \in \Phi_{2D}(S^2)$ represented by their minimal expressions $X_\varphi$ and $X_\psi$, such that $X_\varphi + X_\psi$ has weight at least $2b + 1$. By assumption and using Lemma \ref{lem:2kbound} we find that 
		$$W(X_\varphi), W(X_\psi) \leq 2b - 2z \quad \text{ and } \quad W(X_\varphi + X_\psi) \geq 2k + 2z.$$ 
		Hence, 
		$$4b - 4z \geq W(X_\varphi) + W(X_\psi) \geq W(X_\varphi + X_\psi) \geq 2k+2z.$$ 
		Showing that this can only happen for $2b > k+3z - 1$. 
	\end{proof}
	
	Lemma \ref{lem:2bbound} shows that we can always take $z = 0$ in Proposition \ref{prop:zbound}. This implies that the conjecture must hold for $2b < k$.
	
	\subsection{Technical lemmas} \label{subs:basis} In this subsection we will prove two technical results that we will need for the proof of the main theorem. The philosophy behind them is that we want to be able to split the annihilator of $S^2$ up into forms \textit{at infinity} and forms \textit{at zero}, motivated by Freiman's distinction between left and right stable holes. In practice we will take $f \in S$ such that $(f)_\infty = D$, define $E := (f)_0$ and consider $\varphi \in \divi D_{2D}$ and $\varphi \in \divi E_{2D}$. Both proofs of the lemmas below rely on choosing a specific basis for $L(2D)$ based on $f$ and  $\Phi_{2D}(S^2)$. 
	
	By minimality of $D$ there must exist some $f \in S$ such that $(f)_\infty = D$. Fix such an $f$ and define $E:=(f)_0$. Write $\varphi_1, \ldots, \varphi_b$ for a basis of $\Phi_D(S)$. By Corollary \ref{cor:representative} there are expressions $X_1, \ldots, X_b \in \divi D$ and $Y_1, \ldots, Y_b \in \divi E$ such that 
	$$\varphi_i = X_{i, D} = Y_{i, D}. \qquad \text{ for } \qquad i \in [1, b].$$ 
	Furthermore, since $\dim \divi D = \dim \divi D_D + g$ and $\dim \divi E = \dim \divi E_D + g$, there are $g$ linearly independent expressions $X_{b+1}, \ldots, X_{b+g} \in \divi D$ and $g$ linearly independent expressions $Y_{b+1}, \ldots, Y_{b+g} \in \divi E$ such that $$X_{i, D} = Y_{i, D} = 0 \qquad \text{ for } i \in [b+1, b+g].$$ We denote the translate of $Y_i$ by $f$ as $Y'_i$ for all $i \in [1, b+g]$. 
	
	Define a basis $1 = f_0, f_1, \ldots, f_{N-g} = f$ of $L(D)$ such that 
	$$\varphi_i(f_j) = \delta_{i, j} \qquad \text{ for } \qquad i\in [1, b], \quad j \in [0, N-g],$$
	where $\delta_{i, j}$ denotes the Kronecker-delta symbol. 
	We then have the following for all $j \in [0, N-g]$: 
	\begin{alignat*}{2}
		X_{i,2D}(f_j) &= X_{i, D}(f_j) = \begin{cases} \delta_{i,j} &\text{ for } i \in [1, b] \\ 0 &\text{ for } i \in [b+1, b+g] \end{cases},
		\qquad X_{i,2D}(f f_j) = 0 \text{ for } i \in [1, b+g],
		\\
		Y'_{i,2D}(f_j) &= 0 \text{ for } i \in [1, b+g],
		\qquad Y'_{i,2D}(ff_j) = Y_{i, D}(f_j) = \begin{cases} \delta_{i,j} &\text{ for } i \in [1, b] \\ 0 &\text{ for } i \in [b+1, b+g] \end{cases}.
	\end{alignat*}
	The elements $f_0, \ldots, f_{N-g}, ff_1, \ldots, ff_{N-g}$ are all linearly independent in $L(2D)$. Define $h_1, \ldots, h_g$ such that $f_0, \ldots, f_{N-g}, ff_1, \ldots, ff_{N-g}, h_1, \ldots, h_g$ is a basis for $L(2D)$ and 
	$$X_{i, 2D}(h_j) = \delta_{i, j+b} \text{ for } i\in [1, b+g] \text{ and } j \in [1, g].$$ 
	Without loss of generality we can assume that
	$$Y'_{i, 2D} = X_{i, 2D} \text{ for } i \in [b+1, b+g].$$
	Then in particular we have that 
	$$\Phi_{2D}(S + fS) = \langle X_{1,2D}, \ldots, X_{b+g,2D}, Y'_1, \ldots, Y'_{b,2D} \rangle = \langle X_{1,2D}, \ldots, X_{b,2D}, Y'_{1,2D}, \ldots, Y'_{b+g,2D} \rangle.$$ 
	
	Now let $X' = \sum_{i=b+1}^{b+g} \lambda_{i} X_{i}$ for $\lambda_i \in K$ be a non-zero expression. By definition the projection of $X'$ evaluates as zero on $L(D)$. Lemma \ref{lem:weightLD} then implies that
	$$W(X') \in [N-g + 1, N] \subseteq [2b+1, 2k-1].$$ 
	Using Lemmas \ref{lem:2bbound} and \ref{lem:2kbound}, we deduce that for any such expression $X'$ we have $X'_{2D} \not\in \Phi_{2D}(S^2)$. Because $\Phi_{2D}(S^2)$ is additionally a subset of $\Phi_{2D}(S+fS),$ we find that 
	\begin{equation}\label{eq:anhS2}
		\Phi_{2D}(S^2) \subseteq   \langle X_{1,2D}, \ldots, X_{b,2D}, Y'_{1,2D}, \ldots, Y'_{b,2D} \rangle.
	\end{equation}
	Hence, every form $\varphi$ in $\Phi_{2D}(S^2)$ has an expression of the form $X_\varphi + Y_\varphi$, such that $X_\varphi$ is a linear combinations of the $X_i$ for $i \in [1, b]$ and $Y_\varphi$ is a linear combination of the $Y'_i$ for $i \in [1, b]$. 
	
	We are now ready to prove the first technical lemma. The result can be seen as a generalization of the fact that a hole in $B$ can never be both left and right stable. 
	
	\begin{lemma}\label{lem:dimt1}
		Let $D_X\ge0$ and $D_Y\ge0$ be two disjoint divisors. Suppose that there exists $f\in S$ such that 
		$$(f)_0\ge D_X \qquad\text{and}\qquad (f)_\infty=D \ge D_Y .$$ 
		For $i=1,\ldots,t$ let $X_i\in \divi{D_X}$ and $Y_i\in \divi{D_Y}$ be linearly independent expressions such that $$X_{i,D}=Y_{i,D}.$$ Write $Y'_i$ for the translates of $Y_i$ by $f$ and define 
		$$ V:=\langle X_{1},\ldots,X_{t},Y'_{1},\ldots,Y'_{t}\rangle\quad \text{and}\quad \mathcal E:=\langle Z\in V \mid Z_{2D}\in \Phi_{2D}(S^2)\rangle.$$ Then, $\dim (\mathcal E)\le t$. 
	\end{lemma}
	\begin{proof}
		
		Write $\varphi_i := X_{i,D}$. After replacing the pairs $(X_i,Y_i)$ by suitable linear combinations, we may assume that for some $r \in [1, t]$ we have that $\varphi_1,\ldots,\varphi_r$ are linearly independent and that $\varphi_{r+1}=\cdots=\varphi_t=0.$ We will now define a basis of $L(D)$ as described at the beginning of this subsection. Note that the indexing of the $X_i$ is slightly different, in the sense that we have $X_{i, D} = 0$ for all $i \in [r+1, t]$ rather than $i \in [b+1, b+g]$. Hence, we can find $1 = f_0, f_1, \ldots, f_{N-g} = f$, a basis of $L(D)$ such that $\varphi_i(f_j) = \delta_{i, j}$ for all $i \in [1, r]$ and $j \in [0, N-g]$. 
		
		Then for all $i \in [1, t]$ and all $j \in [0, N-g]$ we have that
		\begin{alignat*}{3}
			X_{i,2D}(f_j) &= X_{i,D}(f_j) &&\quad \textup{and}\quad X_{i,2D}(ff_j) &&= 0;\\
			Y'_{i,2D}(f_j) &= 0 &&\quad \textup{and}\quad Y'_{i,2D}(ff_j) &&= Y_{i,D}(f_j).
		\end{alignat*}
		In particular, for all $i \in [1, t]$ the forms $X_{i,2D}$ and $Y'_{i,2D}$ evaluate zero at $f_j$ and $ff_j$ when $j\in[r+1,N-g]$.
		
		Now suppose there exist linearly independent 
		$$Z_1, \ldots, Z_{t+1} \in \mathcal E.$$ 
		Define $T := L(2D)\cap_{i=1}^{t+1} \Ker(Z_{i, 2D})$, its dimension over $K$ lies in the interval $[2N - g - t, 2N - g - r]$. Since $S^2 \subseteq T$ and $K(f)S^2 = F$, Proposition \ref{prop:KfS2} implies $K(f)T = F$ as well. Hence, $\dim_{K(f)}(T) = N$. 
		
		By our previous observations it must hold that $Z_{i, 2D}(f_j) = Z_{i, 2D}(ff_j)$ for all $ j \in [r+1, N-g]$ and thus,
		$$\langle 1, f_{r+1}, \ldots, f_{N-g}, ff_{r+1}, \ldots, ff_{N-g} \rangle \subseteq T.$$ 
		Over $K(f)$ this space becomes $\langle 1, f_{r+1}, \ldots, f_{N-g} \rangle$ implying that 
		$$\dim_K(T) - \dim_{K(f)}(T) \geq N - g - r + 1.$$ 
		We find
		$$N = \dim_{K(f)} (T) \leq 2N - g - r - N + g + r -1 = N - 1,$$
		a contradiction. 
	\end{proof}
	
	The next lemma gives conditions under which it is possible to actually split forms in $\Phi_{2D}(S^2)$ up into their parts at infinity and at zero. 
	\begin{lemma}\label{lem:dis}
		Let $\varphi \in \Phi_{2D}(S^2)$ and assume that $D \geq D_\varphi$. Take $f \in S$ such that $(f)_\infty = D$ and define $E := (f)_0$. For any form in $\Phi_{2D}(S^2)$ that is represented by $X+Y$ with $X \in \divi {D_\varphi}$ and $Y \in \divi E$ it holds that $$X_{2D}, Y_{2D} \in \Phi_{2D}(S^2).$$
	\end{lemma}
	\begin{proof}
		For $i \in [1, b]$ and $j \in [0, N-g]$ we define $X_i, Y'_i$ and $f_j$ exactly as in the beginning of this subsection. By taking linear combinations of the $X_i$ we may assume that $X_1, \ldots, X_c \in \divi {D_\varphi}$ such that $c$ is maximal. Then Corollary \ref{cor:m_dimension} tells us that $c \geq w/2$, where $w:= w_{2D}(\varphi)$. Furthermore, recall that by \eqref{eq:anhS2} we know that
		$$\Phi_{2D}(S^2) \subseteq  \langle X_{1, 2D}, \ldots, X_{b, 2D}, Y'_{1, 2D}, \ldots, Y'_{b, 2D} \rangle.$$ 
		Let $X$ and $Y$ as in the lemma statement, then the previous implies that 
		$$X = \sum_{i=1}^c \lambda_i X_i \qquad \text{ and } \qquad Y = \sum_{i=1}^{b} \mu_iY'_i \qquad \text{ with } \lambda_i, \mu_i \in K.$$ We will continue the proof by contradiction.
		
		Write $\chi$ for $(X+Y)_{2D}$ and suppose that $X_{2D} \not\in \Phi_{2D}(S^2)$. Since for any $\lambda \in K$ we also have $\chi + \lambda \varphi \in \Phi_{2D}(S^2)$ and thus $X_{2D} + \lambda X_{\varphi, 2D} \not\in \Phi_{2D}(S^2)$, we may assume that $W(X) = w$. We can now redefine $X_1, \ldots, X_c$ such that $X_1 = X$ and $W(X_1) > W(Z)$ for any $Z \in \langle X_2, \ldots, X_c \rangle$. 
		
		Furthermore, note that since $\chi = X_{2D} + Y_{2D}$ it must also hold that $Y_{2D} \not\in \Phi_{2D}(S^2)$ and therefore there is some $i \in [1, b]$ such that \begin{equation}\label{EQ:etabasis}
			\Phi_{2D}(S^2) \subseteq \langle \chi, X_{2, 2D}, \ldots, X_{b, 2D}, Y'_{1, 2D}, \ldots, \widehat{Y'_{i, 2D}}, \ldots, Y'_{b, 2D} \rangle.
		\end{equation}
		
		Let $h \in L(D)$ be the element for which $Y'_{j, 2D}(fh) = 0$ for $j \neq i$ and $Y_{2D}(fh) = 1$. Then all the forms on the right hand side of \eqref{EQ:etabasis} evaluate zero at $f_1 - fh$ and thus, $f_1 - fh \in S^2$. We will use that this implies that $\varphi(f_1-fh) = 0$.  
		
		Define $\nu_i \in K$ such that $\varphi = \sum_{i=1}^c \nu_i X_{i, 2D}$, then
		$$0 = \varphi (f_1 - fh) = \varphi (f_1) = \nu_1 X_{1, 2D}(f_1) = \nu_1.$$ 
		Therefore, $\varphi \in \langle X_{2, 2D}, \ldots, X_{c, 2D} \rangle$, but then $W(\varphi) < W(X_{1}) = w$, a contradiction. 
		
		We conclude that $X_{2D}$ is in $\Phi_{2D}(S^2)$ and therefore so is $\chi - X_{2D} = Y_{2D}$. 
	\end{proof}
	
	\section{Proof of Theorem \ref{Thm:main}}\label{sec:main}
	In this section, we complete the proof of Theorem \ref{Thm:main}. We first recall what has already been proven. By Lemma \ref{lem:2bbound}, every form in $\Phi_{2D}(S^2)$ has weight at most $2b$. Moreover, if every such form has weight at most $2b-2$, then Proposition \ref{prop:zbound}, applied with $z=1$, proves the theorem. Hence, it remains to consider the case where there exists a form in $\Phi_{2D}(S^2)$ of weight $2b-1$ or $2b$. From now on, we fix such a form $\varphi$ and define $w:=w_{2D}(\varphi)$.	
	
	The remaining argument proceeds as follows.
	First, we will show that after translating $S$ and $L(D)$ by some $s^{-1}$ for $s\in S$, we may assume that $D\geq D_\varphi$. This is needed to be able to apply Lemmas \ref{lem:dimt1} and \ref{lem:dis}, for some $f \in S$ with $(f)_\infty = D$. We then prove the existence of a basis of $\Phi_{2D}(S^2)$ whose elements have minimal expressions supported either on $D_\varphi$ or on $E:=(f)_0$.
	
	Finally, mimicking Freiman's argument, we show that the second possibility
	cannot occur. Consequently,
	$$\Phi_{2D}(S^2)\subseteq\Phi_{2D}(2D-D_\varphi),$$
	and therefore
	$$L(2D-D_\varphi)\subseteq S^2.$$
	Since $\deg D_\varphi=w\leq 2b$, we will have
	$$\deg(2D-D_\varphi)=2N-w\geq 2N-2b=2k+2g-2.$$
	Hence, by Theorem \ref{thm:RR},
	$$\ell(2D-D_\varphi)\geq 2k+g-1,$$
	which completes the proof.
	
	\begin{lemma}\label{lem:fixf}
		There exist $t \in S$ such that $\Supp(D + (t))\cap\Supp(D_\varphi) = \varnothing$. For any such $t$ we can find $h \in t^{-1}S$ such that: 
		$$(h)_0 \geq D_\varphi \qquad \text{ and } \qquad (h)_\infty = D  + (t).$$
	\end{lemma}
	\begin{proof}
		For the first statement let $t' \in S$ such that $(t')_\infty = D$, which exists by the minimality of $D$. Note that for any $a \in K$ we then have $(t' + a)_\infty = D$ and thus in particular we can choose some $a$ such that $(t'+a)_0$ is disjoint from $D_\varphi$. Define $t := t' + a$, we find that $D + (t) = (t)_0$ which has the required property. For the rest of the proof we will assume that $t = 1$, by replacing $S$ by $t^{-1}S \subseteq L(D + (t))$ and $\varphi$ by its translate by $t^{-1}$. Then $D$ is replaced by $D + (t)$ and $D_\varphi$ and $w$ stay the same by Lemma \ref{lem:translate}.
		
		Because $w \geq 2b-1$, Corollary \ref{cor:m_dimension} tells us that there exist $b$ linearly independent linear forms
		$$\varphi_1, \ldots, \varphi_{b} \in \Supp(\varphi) \cap \Phi_D(S).$$
		Since the dimension of $\Phi_D(S)$ is $b$ we have
		\begin{equation} \label{EQ:anSb}
			\Phi_D(S) = \langle \varphi_1, \ldots, \varphi_b \rangle \subseteq \Supp_D(\varphi).
		\end{equation}
		It follows that any $f \in L(D)$ such that $(f)_0 \geq D_\varphi$ must be an element of $S$. 
		
		Define $D' = D - D_\varphi$ then $$\deg(D') = N - w \geq k - b + g -1 \geq 2g+2,$$ where the last inequality uses Lemma \ref{lem:gammabound} and $k\geq \gamma +3$. Thus, by Theorem \ref{thm:RR}, we find that for any $P \in \Supp(D)$ we have
		$$\ell(D' - P) = \ell(D') - 1.$$ 
		This shows that there must exist some $h \in L(D')$ that does not lie in any $L(D'-P)$, in other words $(h)_\infty = D$ and $(h)_0 \geq D_\varphi$. Hence we have found $h \in S$ as wished.
	\end{proof}
	\begin{corollary} \label{cor:phis}
		There exists $s \in S$ such that $D + (s) \geq D_\varphi$.
	\end{corollary}
	\begin{proof}
		Let $h$ be such that it meets the conditions of Lemma \ref{lem:fixf} for some $t$. We take $s := ht$, then \begin{equation*}
			D + (s) = D + (t) + (h) = (h)_0 \geq D_\varphi. \qedhere \end{equation*}
	\end{proof}
	
	Since $L(G)\subseteq S$ if and only if $L(G + (s)) \subseteq s^{-1}S$ we may replace $S$ by $s^{-1}S$ and correspondingly $D$ by $D + (s)$ and $\varphi$ by it translate by $s^{-1}$. Note that $D_\varphi$ does not change, since by Lemma \ref{lem:translate} there must exist an expression for $\varphi$ with the same support as before. This also implies that $w$ stays the same. Hence, by Corollary \ref{cor:phis} we can assume that $D \geq D_\varphi$.
	
	\begin{lemma} \label{lem:disj}
		Let $f \in S$ such that $(f)_\infty = D \geq D_\varphi$ and $E := (f)_0$. There exists a basis $\eta_1,\ldots, \eta_{2b + g - \gamma}$ of $\Phi_{2D}(S^2)$ such that the minimal expression for each $\eta_i$ is either in $\divi {D_\varphi}$ or in $\divi E$. 
	\end{lemma}
	\begin{proof}
		By Corollary \ref{cor:m_dimension} we find linearly independent $\varphi_1, \ldots, \varphi_b \in \divi {D_\varphi}_D \cap \Phi_{D}(S)$. Since the dimension of $\Phi_D(S)$ is $b$ we have 
		$$\Phi_D(S) = \langle \varphi_1, \ldots, \varphi_b \rangle \subseteq \divi {D_\varphi}_D.$$ 
		This implies, that if we define the $X_i$ and $Y'_i$ for $i \in [1, b]$ as we did in Subsection \ref{subs:basis}, then all the $X_i$ will be elements of $\divi {D_\varphi}$. Recall that by \eqref{eq:anhS2} we have
		$$\Phi_{2D}(S^2) \subseteq \langle X_{1, 2D}, \ldots, X_{b, 2D}, Y'_{1, 2D}, \ldots, Y'_{b, 2D} \rangle.$$ 
		We deduce that any form in $\Phi_{2D}(S^2)$ can be written as $X + Y$ with $X \in \divi {D_\varphi}$ and $Y \in \divi E$. 
		
		Suppose $\eta_1, \ldots, \eta_{2b + g - \gamma}$ is a basis for $\Phi_{2D}(S^2)$ and let $i$ be the smallest $i$ such that $\eta_i$ is not in $\divi {D_\varphi}_{2D}$ or $\divi E_{2D}$. From the above we know that $\eta_i$ can be expressed as $X + Y$ with $X \in\divi {D_{\varphi}}$ and $Y \in \divi E$. By Lemma \ref{lem:dis} we find that $X_{2D}, Y_{2D} \in \Phi_{2D}(S^2)$. Because $\eta_i$ is linearly independent from $\eta_1, \ldots, \eta_{i-1}$ at least one of $X_{2D}$ and $Y_{2D}$ must be linearly independent from them as well. Hence, we can replace $\eta_i$ by either $X_{2D}$ or $Y_{2D}$. This can be done repeatedly until all $\eta_1, \ldots, \eta_{2b + g - \gamma}$ either have an expression in $D_\varphi$ or in $E$.
		
		Finally, we note that $$\deg D_\varphi, \deg E \leq N < 2k,$$
		which by Lemma \ref{lem:2kbound} implies the expressions must all be minimal. 
	\end{proof}
	
	\begin{proposition}\label{prop:main}
		Let $\eta_1, \ldots, \eta_{2b + g -\gamma}$ be as in Lemma \ref{lem:disj}. The minimal expressions of the elements $\eta_i$ are all in $\divi {D_{\varphi}}$. Hence, $$\Phi_{2D}(S^2) \subseteq \Phi_{2D}(2D - D_\varphi).$$ 
	\end{proposition}
	\begin{proof}
		Suppose the statement is not true, then by Lemma \ref{lem:disj} there exist $\psi \in \Phi_{2D}(S^2)$ with minimal expression in $\divi E$. Because by Corollary \ref{cor:weight2} it holds that $w_{2D}(\psi) \geq 2$ for all $\psi$ we have
		$$w_{2D}(\psi) + w > 2b.$$ 
		Hence, we can choose $\chi \in \Phi_{2D}(S^2)\cap \divi{D_\varphi}_{2D}$ of weight $w_0$ and $\psi \in \Phi_{2D}(S^2) \cap \divi E$ of weight $w_\infty$ such that 
		\begin{equation*}
			w_0 + w_\infty > 2b. \label{eqn:w0_winf}
		\end{equation*}
		Fix such $\chi$ and $\psi$ such that the value $w_0 + w_\infty$ is minimal. We write $X$ and $Y$ for the minimal expressions of $\chi$ and $\psi$, respectively. 
		
		Corollary \ref{cor:m_dimension} tells us that for some $c \geq w_{0}/2$ we have linearly independent forms 
		$$\chi_1, \ldots, \chi_c \in \Phi_D(S) \cap \Supp(X)_D,$$
		and for some $e \geq w_\infty/2$ we have linearly independent forms $$\psi_1, \ldots, \psi_e \in \Phi_D(S) \cap \Supp(Y)_D.$$ 
		Define $x$ such that
		\begin{equation}
			\dim \left(\langle \chi_1, \ldots, \chi_c, \psi_1, \ldots, \psi_e \rangle \right) = c + e - x \label{eqn:chi_psi}.
		\end{equation}
		Then $x$ corresponds to the dimension of the following $K$-subspace
		\begin{equation}
			\langle \chi_1,\ldots,\chi_c\rangle
			\cap
			\langle \psi_1,\ldots,\psi_e\rangle = \left\{ \eta \in L(D)^\ast \mid \eta = \sum_{i=1}^c \lambda_i \chi_i = \sum_{i=1}^e \mu_i \psi_i \text{ for } \lambda_i, \mu_i \in K \right\}.\label{eqn:dim_x}
		\end{equation}
		Since all the forms in \eqref{eqn:chi_psi} lie in the $b$ dimensional space $\Phi_D(S)$, we have $c + e - x \leq b$. We obtain that
		\begin{equation}\label{EQ:bound0}
			w_0 + w_\infty \leq 2c + 2e \leq 2b + 2x.
		\end{equation}
		We now prove two upper bounds on $x$. 
		
		Let $\eta^{(1)}, \ldots, \eta^{(x)}$ be a basis of the space given in \eqref{eqn:dim_x}. Define $X^{(j)}$ and $Y^{(j)}$ as the expressions of $\eta^{(j)}$ in $\divi {D_\varphi}$ and $\divi {E}$, respectively. By taking linear combinations we can assume that
		$$w_0 \geq W(X^{(1)}) > \ldots > W(X^{(x)}).$$ 
		Thus, $W(X^{(x)}) \leq w_0 - x + 1$ and trivially $W(Y^{(x)}) \leq w_\infty$. Furthermore, since 
		$$\left(X^{(x)} - Y^{(x)}\right)_D = 0,$$ 
		Lemma \ref{lem:weightLD} implies that 
		$$W\left(X^{(x)}\right) + W\left(Y^{(x)}\right) \geq N - 2g + 2.$$ 
		Combining this gives
		$$x \leq w_0 + w_\infty - N - 1 + 2g.$$ 
		This is our first upper bound on $x$.
		
		For the second bound, we first show that up to scaling there is at most one expression $Z \in \langle X^{(1)}, \ldots, X^{(x)} \rangle$ that project to a linear form in $\Phi_{2D}(S^2)$. Assume that we have such a non-zero expression $Z$ with $Z_{2D} \in \Phi_{2D}(S^2)$. After replacing the basis $\eta^{(1)}, \ldots, \eta^{(x)}$ if necessary, we may assume that $Z = X^{(j)}$ for some $j$. Note that $Z_{2D} + \psi\in \Phi_{2D}(S^2)$  and recall that $Y_{2D} = \psi$ and $w_\infty \geq W(Y^{(j)})$ for any $j$. Hence,
		$$W(Z) + w_\infty = W(Z+Y) \geq W\left( X^{(j)} + Y^{(j)} \right) \geq N - 2g + 2,$$
		where we used Lemma \ref{lem:weightLD} to estimate the weight of $X^{(j)} + Y^{(j)}$. By Lemma \ref{lem:gammabound} we know that $N-2g+2 = k - g + b + 1 > 2b$. We find that
		$$W(Z) + w_\infty > 2b.$$ 
		Because $w_\infty + w_0$ is minimal, we must have $W(Z) = w_0$ and thus that $Z = X^{(1)}$. Hence,
		$$\dim\left( \langle X^{(1)}_{2D}, \ldots, X^{(x)}_{2D} \rangle \cap \Phi_{2D}(S^2) \right) \leq 1.$$ 
		Analogously we have that 
		$$\dim\left( \langle Y'^{(1)}_{2D}, \ldots, Y'^{(x)}_{2D} \rangle \cap \Phi_{2D}(S^2) \right) \leq 1,$$ 
		where $Y'^{(j)}$ is the translate of $Y^{(j)}$ by $f$. 
		
		Next, we consider an element $\eta$ of 
		$$\Phi_1 := \langle X^{(1)}_{2D}, \ldots, X^{(x)}_{2D}, Y'^{(1)}_{2D}, \ldots, Y'^{(x)}_{2D} \rangle \cap \Phi_{2D}(S^2).$$
		We can write $\eta = \eta_X + \eta_Y$ with $\eta_X \in \langle X^{(1)}_{2D}, \ldots, X^{(x)}_{2D} \rangle$ and $\eta_Y \in \langle Y'^{(1)}_{2D}, \ldots, Y'^{(x)}_{2D} \rangle.$ From Lemma \ref{lem:dis} it follows that both $\eta_X$ and $\eta_Y$ are in $\Phi_{2D}(S^2)$. Therefore, we can add the dimensions of the two spaces to find
		$$\dim \left( \Phi_1 \right) \leq 2.$$ 
		
		Recall that the $X^{(i)}$ represent $\eta^{(i)} \in \Phi_D(S)$. We extend $\eta^{(1)}, \ldots, \eta^{(x)}$ to a basis of $\Phi_D(S)$ by defining $\eta^{(x+1)}, \ldots, \eta^{(b)}$ with two sets of representatives: $X^{(x+1)}, \ldots, X^{(b)} \in \divi {D}$ and $Y^{(x+1)}, \ldots, Y^{(b)} \in \divi {E}$. Let $Y'^{(j)}$ be the translates of $Y^{(j)}$ by $f$. Then we can choose $\eta^{(x+1)}, \ldots, \eta^{(b)}$ such that 
		$$\Phi_{2D}(S^2) = \Phi_1 \oplus \langle X_{x+1,2D}, \ldots, X_{b,2D}, Y'_{x+1,2D}, \ldots, Y'_{b,2D} \rangle \cap \Phi_{2D}(S^2).$$ 
		Since the dimension of $\Phi_{2D}(S^2)$ equals $2b + g - \gamma$ we deduce that
		$$\dim \left(\langle X^{(x+1)}_{2D}, \ldots, X^{(b)}_{2D}, Y'^{(x+1)}_{2D}, \ldots, Y'^{(b)}_{2D} \rangle \cap \Phi_{2D}(S^2) \right) \geq 2b + g - \gamma - 2.$$ 
		By Lemma \ref{lem:dimt1} we also have an upper bound on this dimension of $b-x$. Hence,
		$$x \leq \gamma - b - g + 2.$$ 
		This is our second upper bound for $x$. 
		
		To conclude the proof we combine the two bounds with \eqref{EQ:bound0}. Plugging in both upper bounds for $x$ once we find
		$$w_0 + w_\infty \leq 2b + w_0 + w_\infty - N - 1 + 2g + \gamma - b - g + 2.$$ 
		Hence, 
		$$k = N+1 - b - g \leq \gamma + 2,$$
		which contradicts $k \geq \gamma + 3$. This proves the statement.
	\end{proof}

	\printbibliography
\end{document}